\documentclass[preprint,authoryear]{elsarticle}
\usepackage{amsmath}
\usepackage{amstext}
\usepackage{amsfonts}
\usepackage{amssymb}
\usepackage{graphicx}
\usepackage{bm}
\usepackage{enumitem}

\usepackage{tikz}
\usepackage{pgfplots}

\newcommand{\RR}{{\mathbb{R}}}
\newcommand{\NN}{{\mathbb{N}}}

\newcommand{\CC}{{\mathbb{C}}}

\newcommand{\asto}{\overset{\rm a.s.}{\longrightarrow}}
\newcommand{\EE}{{\rm E}}

\DeclareMathOperator{\tr}{tr}
\DeclareMathOperator{\argmin}{argmin}
\DeclareMathOperator{\diag}{diag}

\newcounter{ctheorem}
\newtheorem{theorem}[ctheorem]{Theorem}
\newcounter{cassumption}
\newtheorem{assumption}[cassumption]{Assumption}

\newproof{proof}{Proof}

\newcounter{cproposition}
\newtheorem{proposition}[cproposition]{Proposition}
\newcounter{ccorollary}
\newtheorem{corollary}[ccorollary]{Corollary}
\newcounter{clemma}
\newtheorem{lemma}[clemma]{Lemma}

\journal{Journal of Multivariate Analysis}

\begin{document}

\begin{frontmatter}

\title{Large Dimensional Analysis and Optimization of Robust Shrinkage Covariance Matrix Estimators\tnoteref{t1}}
\tnotetext[t1]{ Couillet's work is supported by the ERC MORE EC--120133. McKay's work is supported by the Hong Kong Research Grants Council under grant number 616713.}



\author[supelec]{Romain Couillet}
\ead{romain.couillet@supelec.fr}
\address[supelec]{Telecommunication department, Sup\'elec, Gif sur Yvette, France}

\author[HKUST]{Matthew McKay}
\ead{eemckay@ust.hk}
\address[HKUST]{Hong Kong University of Science and Technology}

%




\begin{abstract}
	This article studies two regularized robust estimators of scatter matrices proposed (and proved to be well defined) in parallel in \citep{CHE11} and \citep{PAS13}, based on Tyler's robust M-estimator \citep{TYL87} and on Ledoit and Wolf's shrinkage covariance matrix estimator \citep{LED04}. These hybrid estimators have the advantage of conveying (i) robustness to outliers or impulsive samples and (ii) small sample size adequacy to the classical sample covariance matrix estimator. We consider here the case of i.i.d. elliptical zero mean samples in the regime where both sample and population sizes are large. We demonstrate that, under this setting, the estimators under study asymptotically behave similar to well-understood random matrix models. This characterization allows us to derive optimal shrinkage strategies to estimate the population scatter matrix, improving significantly upon the empirical shrinkage method proposed in \citep{CHE11}.
\end{abstract}	


\end{frontmatter}


\section{Introduction}

Many scientific domains customarily deal with (possibly small) sets of large dimensional data samples from which statistical inference is performed. This is in particular the case in financial data analysis where few stationary monthly observations of numerous stock indexes are used to estimate the joint covariance matrix of the stock returns \citep{POT00,LED03,RUB12}, bioinformatics where clustering of genes is obtained based on gene sequences sampled from a small population \citep{SCH05}, computational immunology where correlations among mutations in viral strains are estimated from sampled viral sequences and used as a basis of novel vaccine design \citep{DAH11,QUA13}, psychology where the covariance matrix of multiple psychological traits is estimated from data collected on a group of tested individuals \citep{STE80}, or electrical engineering at large where signal samples extracted from a possibly short time window are used to retrieve parameters of the signal \citep{SCH91}. In many such cases, the number $n$ of independent data samples $x_1,\ldots,x_n\in \CC^N$ (or $\RR^N$) may not be large compared to the size $N$ of the population, suggesting that the empirical sample covariance matrix $\bar{C}_N=\frac1n\sum_{i=1}^n(x_i-\bar{x})(x_i-\bar{x})^*$, $\bar{x}=\frac1n\sum_{i=1}^nx_i$, is a poor estimate for $C_N=\EE[(x_1-\EE[x_1])(x_1-\EE[x_1])^*]$. Several solutions have been proposed to work around this problem. If the end application is not to retrieve $C_N$ but some metric of it, recent works on random matrix theory showed that replacing $C_N$ in the metric by $\bar{C}_N$ often leads to a biased estimate of the metric \citep{MES08}, but that this estimate can be corrected by an improved estimation of the metric itself via the samples $x_1,\ldots,x_n$ \citep{MES08b}. However, when the object under interest is $C_N$ itself and $N\simeq n$, there is little hope to retrieve any consistent estimate of $C_N$. A popular alternative proposed originally in \citep{LED04} is to ``shrink'' $\bar{C}_N$, i.e., consider instead $\bar{C}_N(\rho)=(1-\rho)\bar{C}_N+\rho I_N$ for an appropriate $\rho\in [0,1]$ that minimizes the average distance $\EE[\tr( (\bar{C}_N(\rho)-C_N)^2)]$. The interest of $\rho$ here is to give more or less weight to $\bar{C}_N$ depending on the relevance of the $n$ samples, so that in particular $\rho$ is better chosen close to zero when $n$ is large and close to one when $n$ is small.

In addition to the problem of scarcity of samples, it is often the case that outliers are present among the set of samples. These outliers may arise from erroneous or inconsistent data (e.g., individuals under psychological or biological tests incorrectly identified to fit the test pattern), or from the corruption of some samples by external events (e.g., interference by ambient electromagnetic noise in signal processing). These outliers, if not correctly handled, may further corrupt the statistical inference and in particular the estimation of $C_N$. The field of robust estimation intends to deal with this problem \citep{HUB81,MAR06} by proposing estimators that have the joint capability to naturally attenuate the effect of outliers \citep{HUB64} as well as to appropriately handle samples of an impulsive nature \citep{TYL87}, e.g., elliptically distributed data. A common denominator of such estimators is their belonging to the class of M-estimators, therefore taking the form of the solution to an implicit equation. This poses important problems of analysis in small $N,n$ dimensions, resulting mostly in only asymptotic results in the regime $N$ fixed and $n\to\infty$ \citep{MAR76,KEN91}. This regime is however inconsistent with the present scenario of scarce data where $N\simeq n$. Nonetheless, recent works based on random matrix theory have shown that a certain family of such robust covariance matrix estimators asymptotically behave as $N,n\to\infty$ and $N/n\to c\in(0,1)$ similar to classical random matrices taking (almost) explicit forms. Such observations were made for the class of Maronna's M-estimators of scatter \citep{MAR76} for sample vectors whose independent entries can contain outliers \citep{COU13} and for elliptically distributed samples \citep{COU13b}, as well as for Tyler's M-estimator \citep{TYL87} in \citep{ZHA14}.

In this article, we study two hybrid robust shrinkage covariance matrix estimates $\hat{C}_N(\rho)$ (hereafter referred to as the Abramovich--Pascal estimate) and $\check{C}_N(\rho)$ (hereafter referred to as the Chen estimate) proposed in parallel in \citep{ABR07,PAS13}\footnote{To the authors' knowledge, the first instance of the estimator dates back to \citep{ABR07} although the non-obvious proof of $\hat{C}_N(\rho)$ being well-defined is only found later in \citep{PAS13}.} and in \citep{CHE11}, respectively. Both matrices, whose definition is introduced in Section~\ref{sec:results} below, are empirically built upon Tyler's M-estimate \citep{TYL87} originally designed to cope with elliptical samples whose distribution is unknown to the experimenter and upon the Ledoit--Wolf shrinkage estimator \citep{LED04}. This allows for an improved degree of freedom for approximating the population covariance matrix and importantly allows for $N>n$, which Maronna's and Tyler's estimators do not. In \citep{PAS13} and \citep{CHE11}, $\hat{C}_N(\rho)$ and $\check{C}_N(\rho)$ were proved to be well-defined as the unique solutions to their defining fixed-point matrices. However, little is known of their performance as estimators of $C_N$ in the regime $N\simeq n$ of interest here. Some progress in this direction was made in \citep{CHE11} but this work does not manage to solve the optimal shrinkage problem consisting of finding $\rho$ such that $\EE[\tr((\check{C}_N(\rho)-C_N)^2)]$ is minimized and resorts to solving an approximate problem instead. 

The present article studies the matrices $\hat{C}_N(\rho)$ and $\check{C}_N(\rho)$ from a random matrix approach, i.e., in the regime where $N,n\to\infty$ with $N/n\to c\in(0,\infty)$, and under the assumption of the absence of outliers. Our main results are as follows:
\begin{itemize}
	\item we show that, under the aforementioned setting, both $\hat{C}_N(\rho)$ and $\check{C}_N(\rho)$ asymptotically behave similar to well-known random matrix models and prove in particular that both have a well-identified limiting spectral distribution;
	\item we prove that, up to a change in the variable $\rho$, the matrices $\check{C}_N(\rho)$ and $\hat{C}_N(\rho)/(\frac1N\tr \hat{C}_N(\rho))$ are essentially the same for $N,n$ large, implying that both achieve the same optimal shrinkage performance;
	\item we determine the optimal shrinkage parameters $\hat\rho^\star$ and $\check\rho^\star$ that minimize the almost sure limits $\lim_N\frac1N\tr[(\hat{C}_N(\rho)/(\frac1N\tr \hat{C}_N(\rho))-C_N)]^2$ and $\lim_N\frac1N\tr[(\check{C}_N(\rho)-C_N)]^2$, respectively, both limits being the same. We then propose consistent estimates $\hat\rho_N$ and $\check\rho_N$ for $\hat\rho^\star$ and $\check\rho^\star$ which achieve the same limiting performance. We finally show by simulations that a significant gain is obtained using $\hat\rho^\star$ (or $\hat\rho_N$) and $\check\rho^\star$ (or $\check\rho_N$) compared to the solution $\check\rho_O$ of the approximate problem developed in \citep{CHE11}.
\end{itemize}
In practice, these results allow for a proper use of $\hat{C}_N(\rho)$ and $\check{C}_N(\rho)$ in anticipation of the absence of outliers. In the presence of outliers, it is then expected that both Abramovich--Pascal and Chen estimates will exhibit robustness properties that their asymptotic random matrix equivalents will not. Note in particular that, although $\hat{C}_N(\rho)$ and $\check{C}_N(\rho)$ are shown to be asymptotically equivalent in the absence of outliers, it is not clear at this point whether one of the two estimates will show better performance in the presence of outliers. The study of this interesting scenario is left to future work. 

The remainder of the article is structured as follows. In Section~\ref{sec:results}, we introduce our main results on the large $N,n$ behavior of the matrices $\hat{C}_N(\rho)$ and $\check{C}_N(\rho)$. In Section~\ref{sec:shrink}, we develop the optimal shrinkage analysis, providing in particular asymptotically optimal empirical shrinkage strategies. Concluding remarks are provided in Section~\ref{sec:conclusion}. All proofs of the results of Section~\ref{sec:results} and Section~\ref{sec:shrink} are then presented in Section~\ref{sec:proofs}.

{\it General notations:} The superscript $(\cdot)^*$ stands for Hermitian transpose in the complex case or transpose in the real case. The notation $\Vert\cdot\Vert$ stands for the spectral norm for matrices and the Euclidean norm for vectors. The Dirac measure at point $x$ is denoted ${\bm\delta}_x$. The ordered eigenvalues of a Hermitian (or symmetric) matrix $X$ of size $N\times N$ are denoted $\lambda_1(X)\leq \ldots\leq \lambda_N(X)$. For $\ell>0$ and a positive and positively supported measure $\nu$, we define $M_{\nu,\ell}=\int t^\ell \nu(dt)$ (may be infinite). The arrow ``$\asto$'' designates almost sure convergence. The statement $X \triangleq Y$ defines the new notation $X$ as being equal to $Y$.

\section{Main results}
\label{sec:results}

We start by introducing the main assumptions of the data model under study. We consider $n$ sample vectors $x_1,\ldots,x_n\in\CC^N$ (or $\RR^N$) having the following characteristics.

\begin{assumption}[Growth rate]
	\label{ass:c}
	Denoting $c_N=N/n$, $c_N\to c\in (0,\infty)$ as $N\to\infty$.
\end{assumption}

\begin{assumption}[Population model]
	\label{ass:x}
	The vectors $x_1,\ldots,x_n\in\CC^N$ (or $\RR^N$) are independent with
	\begin{enumerate}[label=\alph*.]
		\item \label{item:x} $x_i=\sqrt{\tau}_iA_Ny_i$, where $y_i\in\CC^{\bar N}$ (or $\RR^{\bar N}$), $\bar{N}\geq N$, is a random zero mean unitarily (or orthogonally) invariant vector with norm $\Vert y_i\Vert^2=\bar N$, $A_N\in\CC^{N\times \bar{N}}$ is deterministic, and $\tau_1,\ldots,\tau_n$ is a collection of positive scalars. We shall denote $z_i=A_Ny_i$.
		\item \label{item:C} $C_N\triangleq A_NA_N^*$ is nonnegative definite, with trace $\frac1N\tr C_N=1$ and spectral norm satisfying $\limsup_N \Vert C_N\Vert <\infty$.
		\item \label{item:nu} $\nu_N\triangleq \frac1N\sum_{i=1}^N {\bm\delta}_{\lambda_i(C_N)}$ satisfies $\nu_N\to \nu$ weakly with $\nu\neq {\bm\delta}_0$ almost everywhere.
	\end{enumerate}
\end{assumption}

Since all considerations to come are equally valid over $\CC$ or $\RR$, we will consider by default that $x_1,\ldots,x_n\in\CC^N$. As the analysis will show, the positive scalars $\tau_i$ have no impact on the robust covariance estimates; with this definition, the distribution of the vectors $x_i$ contains in particular the class of elliptical distributions. Note that the assumption that $y_i$ is zero mean unitarily invariant with norm $\bar N$ is equivalent to saying that $y_i=\sqrt{\bar{N}}\frac{\tilde{y}_i}{\Vert \tilde{y}_i\Vert}$ with $\tilde{y}_i\in\CC^{\bar N}$ standard Gaussian. This, along with $A_N\in\CC^{N\times \bar{N}}$ and $\limsup_N \Vert C_N\Vert <\infty$, implies in particular that $\Vert x_i\Vert^2$ is of order $N$. The assumption that $\nu\neq {\bm\delta}_0$ almost everywhere avoids the degenerate scenario where an overwhelming majority of the eigenvalues of $C_N$ tend to zero, whose practical interest is quite limited. Finally note that the constraint $\frac1N\tr C_N=1$ is inconsequential and in fact defines uniquely both terms in the product $\tau_iC_N$.

The following two theorems introduce the robust shrinkage estimators $\hat{C}_N(\rho)$ and $\check{C}_N(\rho)$, and constitute the main technical results of this article.
\begin{theorem}[Abramovich--Pascal Estimate]
	\label{th:Chitour}
	Let Assumptions~\ref{ass:c} and \ref{ass:x} hold. For $\varepsilon\in (0,\min\{1,c^{-1}\})$, define $\hat{\mathcal R}_\varepsilon=[\varepsilon+\max\{0,1-c^{-1}\},1]$. For each $\rho\in (\max\{0,1-c_N^{-1}\},1]$, let $\hat{C}_N(\rho)$ be the unique solution to
	\begin{align*}
		\hat{C}_N(\rho) &= (1-\rho) \frac1n\sum_{i=1}^n \frac{x_ix_i^*}{\frac1Nx_i^*\hat{C}_N(\rho)^{-1}x_i} + \rho I_N.
	\end{align*}
	Then, as $N\to\infty$, 
	\begin{align*}
		\sup_{\rho\in \hat{\mathcal R}_\varepsilon}\left\Vert \hat{C}_N(\rho) - \hat{S}_N(\rho) \right\Vert \asto 0
	\end{align*}
	where
	\begin{align*}
		\hat{S}_N(\rho) &=  \frac1{\hat{\gamma}(\rho)} \frac{1-\rho}{1-(1-\rho)c} \frac1n\sum_{i=1}^n z_iz_i^* + \rho I_N
	\end{align*}
	and ${\hat{\gamma}}(\rho)$ is the unique positive solution to the equation in $\hat{\gamma}$
	\begin{align*}
		1 &= \int \frac{t}{{\hat{\gamma}}\rho + (1-\rho)t}\nu(dt).
	\end{align*}
	Moreover, the function $\rho\mapsto\hat\gamma(\rho)$ thus defined is continuous on $(0,1]$.
\end{theorem}
\begin{proof}
	The proof is deferred to Section~\ref{sec:proof_th_Chitour}.
\end{proof}

\begin{theorem}[Chen Estimate]
	\label{th:Wiesel}
	Let Assumptions~\ref{ass:c} and \ref{ass:x} hold. For $\varepsilon\in (0,1)$, define $\check{\mathcal R}_\varepsilon=[\varepsilon,1]$. For each $\rho\in(0,1]$, let $\check{C}_N(\rho)$ be the unique solution to 
	\begin{align*}
		\check{C}_N(\rho) &=  \frac{\check{B}_N(\rho)}{\frac1N\tr \check{B}_N(\rho)}
	\end{align*}
	where
	\begin{align*}
		\check{B}_N(\rho) &= (1-\rho) \frac1n\sum_{i=1}^n \frac{x_ix_i^*}{\frac1Nx_i^*\check{C}_N(\rho)^{-1}x_i} + \rho I_N.
	\end{align*}
	Then, as $N\to\infty$,
	\begin{align*}
		\sup_{\rho \in \check{\mathcal R}_\varepsilon}\left\Vert \check{C}_N(\rho) - \check{S}_N(\rho) \right\Vert \asto 0
	\end{align*}
	where
	\begin{align*}
		\check{S}_N(\rho) &=  \frac{1-\rho}{1-\rho+T_\rho} \frac1n\sum_{i=1}^n z_iz_i^* + \frac{T_\rho}{1-\rho+T_\rho} I_N
	\end{align*}
	in which $T_\rho=\rho\check{\gamma}(\rho)F(\check{\gamma}(\rho);\rho)$ with, for all $x>0$,
	\begin{align*}
		F(x;\rho) &= \frac12\left(\rho - c(1-\rho) \right) + \sqrt{\frac14\left(\rho - c(1-\rho) \right)^2+(1-\rho)\frac1x}
	\end{align*}
	and ${\check{\gamma}}(\rho)$ is the unique positive solution to the equation in $\check{\gamma}$
	\begin{align*}
		1 &= \int \frac{t}{ {\check{\gamma}}\rho + \frac{1-\rho}{(1-\rho)c+F({\check{\gamma}};\rho)} t}\nu(dt).
	\end{align*}
	Moreover, the function $\rho\mapsto\check\gamma(\rho)$ thus defined is continuous on $(0,1]$.
\end{theorem}
\begin{proof}
	The proof is deferred to Section~\ref{sec:proof_th_Wiesel}.
\end{proof}

Theorem~\ref{th:Chitour} and Theorem~\ref{th:Wiesel} show that, as $N,n\to\infty$ with $N/n\to c$, the matrices $\hat{C}_N(\rho)$ and $\check{C}_N(\rho)$, defined as the non-trivial solution of fixed-point equations, behave similar to matrices $\hat{S}_N(\rho)$ and $\check{S}_N(\rho)$, respectively, whose characterization is well-known and much simpler than that of  $\hat{C}_N(\rho)$ and $\check{C}_N(\rho)$ themselves. Indeed, $\hat{S}_N(\rho)$ and $\check{S}_N(\rho)$ are random matrices of the sample covariance matrix type thoroughly studied in e.g., \citep{MAR67,SIL95,CHO95}. 

Note that these results are similar in statement to the results of \citep{COU13,COU13b} for robust estimators of the Maronna-type. Technically speaking, the proof of both Theorem~\ref{th:Chitour} and Theorem~\ref{th:Wiesel} unfold from the same technique as produced in these articles. However, while the proof of Theorem~\ref{th:Chitour} comes with no major additional difficulty compared to these works, due to the scale normalization imposed in the definition of $\check{C}_N(\rho)$, the proof of Theorem~\ref{th:Wiesel} requires a more elaborate approach than used in \citep{COU13b}. Another difference to previous works lies here in that, unlike Maronna's estimator that only attenuates the effect of the scale parameters $\tau_i$, the proposed Tyler-based estimators discard this effect altogether. Also, the technical study of Maronna's estimator can be made under the assumption that $C_N=I_N$ (from a natural variable change) while here, because of the regularization term $\rho I_N$, $C_N$ does intervene in an intricate manner in the results.

As a side remark, it is shown in \citep{PAS13} that for each $N,n$ fixed with $n\geq N+1$, $\hat{C}_N(\rho)\to \hat{C}_N(0)$ as $\rho\to 0$ with $\hat{C}_N(0)$ defined (almost surely) as one of the (uncountably many) solutions to
\begin{align}
	\label{eq:CN0}
\hat{C}_N(0)=\frac1n\sum_{i=1}^n \frac{x_ix_i^*}{\frac1Nx_i^*\hat{C}_N(0)^{-1}x_i}.
\end{align}
In the regime where $N,n\to\infty$ and $N/n\to c$, this result is difficult to generalize as it is challenging to handle the limit $\Vert \hat{C}_N(\rho_N)-\hat{S}_N(\rho_N)\Vert$ for a sequence $\{\rho_N\}_{N=1}^\infty$ with $\rho_N\to 0$. The requirement that $\rho_N\to \rho_0>0$ on any such sequence is indeed at the core of the proof of Theorem~\ref{th:Chitour} (see Equations~\eqref{eq:e+} and \eqref{eq:1} in Section~\ref{sec:proof_th_Chitour} where $\rho_0>0$ is necessary to ensure $e^+<1$). This explains why the set $\hat{\mathcal R}_\varepsilon$ in Theorem~\ref{th:Chitour} excludes the region $[0,\varepsilon)$. Similar arguments hold for $\check{C}_N(\rho)$. As a matter of fact, the behavior of any solution $\hat{C}_N(0)$ to \eqref{eq:CN0} in the large $N,n$ regime, recently derived in \citep{ZHA14}, remains difficult to handle with our proof technique.

	An immediate consequence of Theorem~\ref{th:Chitour} and Theorem~\ref{th:Wiesel} is that the empirical spectral distributions of $\hat{C}_N(\rho)$ and $\check{C}_N(\rho)$ converge to the well-known respective limiting distributions of $\hat{S}_N(\rho)$ and $\check{S}_N(\rho)$, characterized in the following result.

\begin{corollary}[Limiting spectral distribution]
	\label{cor:limit}
	Under the settings of Theorem~\ref{th:Chitour} and Theorem~\ref{th:Wiesel},
	\begin{align*}
		\frac1N\sum_{i=1}^N {\bm\delta}_{\lambda_i(\hat{C}_N(\rho))} &\asto \hat{\mu}_\rho,~\rho\in \hat{\mathcal R}_\varepsilon \\
		\frac1N\sum_{i=1}^N {\bm\delta}_{\lambda_i(\check{C}_N(\rho))} &\asto \check{\mu}_\rho,~\rho\in \check{\mathcal R}_\varepsilon
	\end{align*}
	where the convergence arrow is understood as the weak convergence of probability measures, for almost every sequence $\{x_1,\ldots,x_n\}_{n=1}^\infty$, and where
	\begin{align*}
		\hat{\mu}_\rho &= \max\{0,1-c^{-1}\}{\bm\delta}_{\rho} + \underline{\hat{\mu}}_\rho \\
		\check{\mu}_\rho &= \max\{0,1-c^{-1}\}{\bm\delta}_{\frac{ T_\rho }{ 1-\rho + T_\rho } } + \underline{\check{\mu}}_\rho
	\end{align*}
	with $\underline{\hat{\mu}}_\rho$ and $\underline{\check{\mu}}_\rho$ continuous finite measures with compact support in $[\rho,\infty)$ and $[{ T_\rho }( 1-\rho + T_\rho )^{-1},\infty)$ respectively, real analytic wherever their density is positive. The measure $\hat{\mu}_\rho$ is the only measure with Stieltjes transform $m_{\hat{\mu}_\rho}(z)$ defined, for $z\in\CC$ with $\Im[z]>0$, as
	\begin{align*}
		m_{\hat{\mu}_\rho}(z) &= \hat{\gamma}\frac{1-(1-\rho)c}{1-\rho} \int \frac1{ \hat{z}(\rho)  +\frac{t}{1+c \hat{\delta}(z)}}\nu(dt)
	\end{align*}
	where $\hat{z}(\rho)=(\rho-z)\hat{\gamma}(\rho)\frac{1-(1-\rho)c}{1-\rho}$ and $\hat{\delta}(z)$ is the unique solution with positive imaginary part of the equation in $\hat{\delta}$
	\begin{align*}
		\hat{\delta} &= \int \frac{t}{ \hat{z}(\rho)  +\frac{t}{1+c \hat{\delta}}}\nu(dt).
	\end{align*}
	The measure $\check{\mu}_\rho$ is the only measure with Stieltjes transform $m_{\check{\mu}_\rho}(z)$ defined, for $\Im[z]>0$ as
	\begin{align*}
		m_{\check{\mu}_\rho}(z) &= \frac{1-\rho+T_\rho}{1-\rho} \int \frac1{\check{z}(\rho)+\frac{t}{1+c\check{\delta}(z)}}\nu(dt)
	\end{align*}
	with $\check{z}(\rho)=\frac{1}{1-\rho}T_\rho(1-z)-z$ and $\check{\delta}(z)$ the unique solution with positive imaginary part of the equation in $\check{\delta}$
	\begin{align*}
		\check{\delta} &= \int \frac{t}{\check{z}(\rho)+\frac{t}{1+c\check{\delta}}}\nu(dt).
	\end{align*}
\end{corollary}
\begin{proof}
	This is an immediate application of \citep{SIL95,CHO95} and Theorems~\ref{th:Chitour} and \ref{th:Wiesel}.
\end{proof}

\begin{figure}[h!]
  \centering
  \begin{tikzpicture}[font=\footnotesize]
    \renewcommand{\axisdefaulttryminticks}{4} 
    \tikzstyle{every major grid}+=[style=densely dashed]       
    \tikzstyle{every axis y label}+=[yshift=-10pt] 
    \tikzstyle{every axis x label}+=[yshift=5pt]
    \tikzstyle{every axis legend}+=[cells={anchor=west},fill=white,
        at={(0.98,0.98)}, anchor=north east, font=\scriptsize]
    \begin{axis}[
      xmin=0,
      ymin=0,
      xmax=4.5,
      ymax=1.6,
      bar width=1.5pt,
      grid=major,
      ymajorgrids=false,
      scaled ticks=true,
      xlabel={Eigenvalues},
      ylabel={Density}
      ]
      \addplot+[ybar,mark=none,color=black,fill=gray!40!white,area legend] coordinates{
	      (0.025000,0.000000)(0.075000,0.000000)(0.125000,0.000000)(0.175000,0.000000)(0.225000,0.000000)(0.275000,0.000000)(0.325000,0.000000)(0.375000,0.000000)(0.425000,0.078125)(0.475000,1.015625)(0.525000,1.406250)(0.575000,1.484375)(0.625000,1.484375)(0.675000,1.171875)(0.725000,1.250000)(0.775000,1.015625)(0.825000,0.625000)(0.875000,0.468750)(0.925000,0.000000)(0.975000,0.000000)(1.025000,0.000000)(1.075000,0.000000)(1.125000,0.000000)(1.175000,0.000000)(1.225000,0.000000)(1.275000,0.000000)(1.325000,0.000000)(1.375000,0.000000)(1.425000,0.000000)(1.475000,0.000000)(1.525000,0.000000)(1.575000,0.000000)(1.625000,0.078125)(1.675000,0.156250)(1.725000,0.234375)(1.775000,0.156250)(1.825000,0.234375)(1.875000,0.312500)(1.925000,0.312500)(1.975000,0.234375)(2.025000,0.234375)(2.075000,0.312500)(2.125000,0.234375)(2.175000,0.234375)(2.225000,0.312500)(2.275000,0.312500)(2.325000,0.312500)(2.375000,0.312500)(2.425000,0.234375)(2.475000,0.234375)(2.525000,0.312500)(2.575000,0.234375)(2.625000,0.312500)(2.675000,0.156250)(2.725000,0.312500)(2.775000,0.312500)(2.825000,0.234375)(2.875000,0.234375)(2.925000,0.234375)(2.975000,0.234375)(3.025000,0.234375)(3.075000,0.234375)(3.125000,0.234375)(3.175000,0.156250)(3.225000,0.156250)(3.275000,0.156250)(3.325000,0.234375)(3.375000,0.156250)(3.425000,0.156250)(3.475000,0.156250)(3.525000,0.156250)(3.575000,0.156250)(3.625000,0.156250)(3.675000,0.078125)(3.725000,0.156250)(3.775000,0.078125)(3.825000,0.078125)(3.875000,0.156250)(3.925000,0.078125)(3.975000,0.000000)(4.025000,0.000000)(4.075000,0.000000)(4.125000,0.000000)(4.175000,0.000000)(4.225000,0.000000)(4.275000,0.000000)(4.325000,0.000000)(4.375000,0.000000)(4.425000,0.000000)(4.475000,0.000000)(4.525000,0.000000)(4.575000,0.000000)(4.625000,0.000000)(4.675000,0.000000)(4.725000,0.000000)(4.775000,0.000000)(4.825000,0.000000)(4.875000,0.000000)(4.925000,0.000000)
      };
      \addplot[black,smooth,line width=1pt] plot coordinates{
	      (0.,0.000045)(0.02,0.000048)(0.04,0.000051)(0.06,0.000055)(0.08,0.000059)(0.1,0.000064)(0.12,0.000070)(0.14,0.000076)(0.16,0.000084)(0.18,0.000092)(0.2,0.000102)(0.22,0.000114)(0.24,0.000129)(0.26,0.000147)(0.28,0.000169)(0.3,0.000197)(0.32,0.000234)(0.34,0.000286)(0.36,0.000359)(0.38,0.000475)(0.4,0.000678)(0.42,0.001144)(0.44,0.003820)(0.46,0.849217)(0.48,1.172947)(0.5,1.338276)(0.52,1.427212)(0.54,1.472072)(0.56,1.486673)(0.58,1.481180)(0.6,1.460377)(0.62,1.428754)(0.64,1.387955)(0.66,1.339967)(0.68,1.286662)(0.7,1.227588)(0.72,1.164220)(0.74,1.095522)(0.76,1.022960)(0.78,0.944484)(0.8,0.859677)(0.82,0.767183)(0.84,0.664717)(0.86,0.545014)(0.88,0.394253)(0.9,0.128598)(0.92,0.000768)(0.94,0.000451)(0.96,0.000326)(0.98,0.000256)(1.,0.000210)(1.02,0.000178)(1.04,0.000154)(1.06,0.000135)(1.08,0.000121)(1.1,0.000109)(1.12,0.000099)(1.14,0.000091)(1.16,0.000085)(1.18,0.000079)(1.2,0.000074)(1.22,0.000070)(1.24,0.000067)(1.26,0.000064)(1.28,0.000062)(1.3,0.000061)(1.32,0.000059)(1.34,0.000058)(1.36,0.000058)(1.38,0.000058)(1.4,0.000059)(1.42,0.000060)(1.44,0.000062)(1.46,0.000065)(1.48,0.000069)(1.5,0.000075)(1.52,0.000085)(1.54,0.000099)(1.56,0.000126)(1.58,0.000197)(1.6,0.036046)(1.62,0.091694)(1.64,0.121931)(1.66,0.144766)(1.68,0.162967)(1.7,0.178229)(1.72,0.191142)(1.74,0.202434)(1.76,0.212423)(1.78,0.221027)(1.8,0.228507)(1.82,0.235404)(1.84,0.241556)(1.86,0.246773)(1.88,0.251752)(1.9,0.256116)(1.92,0.259965)(1.94,0.263347)(1.96,0.266360)(1.98,0.269002)(2.,0.271524)(2.02,0.273608)(2.04,0.275437)(2.06,0.276992)(2.08,0.278234)(2.1,0.279391)(2.12,0.280449)(2.14,0.281104)(2.16,0.281786)(2.18,0.282080)(2.2,0.282473)(2.22,0.282510)(2.24,0.282594)(2.26,0.282362)(2.28,0.282083)(2.3,0.281757)(2.32,0.281341)(2.34,0.280779)(2.36,0.280123)(2.38,0.279394)(2.4,0.278429)(2.42,0.277627)(2.44,0.276616)(2.46,0.275498)(2.48,0.274452)(2.5,0.273252)(2.52,0.271986)(2.54,0.270530)(2.56,0.269147)(2.58,0.267784)(2.6,0.266210)(2.62,0.264676)(2.64,0.263197)(2.66,0.261441)(2.68,0.259801)(2.7,0.258067)(2.72,0.256335)(2.74,0.254518)(2.76,0.252654)(2.78,0.250742)(2.8,0.248722)(2.82,0.246730)(2.84,0.244832)(2.86,0.242779)(2.88,0.240665)(2.9,0.238481)(2.92,0.236302)(2.94,0.234195)(2.96,0.231933)(2.98,0.229710)(3.,0.227448)(3.02,0.225076)(3.04,0.222717)(3.06,0.220251)(3.08,0.217836)(3.1,0.215481)(3.12,0.212920)(3.14,0.210411)(3.16,0.207936)(3.18,0.205220)(3.2,0.202711)(3.22,0.199948)(3.24,0.197243)(3.26,0.194575)(3.28,0.191884)(3.3,0.188975)(3.32,0.186231)(3.34,0.183202)(3.36,0.180388)(3.38,0.177406)(3.4,0.174355)(3.42,0.171335)(3.44,0.168253)(3.46,0.165089)(3.48,0.161785)(3.5,0.158556)(3.52,0.155100)(3.54,0.151703)(3.56,0.148310)(3.58,0.144853)(3.6,0.141193)(3.62,0.137371)(3.64,0.133705)(3.66,0.129701)(3.68,0.125698)(3.7,0.121596)(3.72,0.117360)(3.74,0.113159)(3.76,0.108503)(3.78,0.103743)(3.8,0.098811)(3.82,0.093671)(3.84,0.088229)(3.86,0.082513)(3.88,0.076294)(3.9,0.069647)(3.92,0.062258)(3.94,0.053960)(3.96,0.044130)(3.98,0.031463)(4.,0.006866)(4.02,0.000084)(4.04,0.000049)(4.06,0.000038)(4.08,0.000032)(4.1,0.000027)(4.12,0.000024)(4.14,0.000022)(4.16,0.000020)(4.18,0.000019)(4.2,0.000017)(4.22,0.000016)(4.24,0.000015)(4.26,0.000014)(4.28,0.000014)(4.3,0.000013)(4.32,0.000012)(4.34,0.000012)(4.36,0.000011)(4.38,0.000011)(4.4,0.000010)(4.42,0.000010)(4.44,0.000009)(4.46,0.000009)(4.48,0.000009)(4.5,0.000009)(4.52,0.000008)(4.54,0.000008)(4.56,0.000008)(4.58,0.000008)(4.6,0.000007)(4.62,0.000007)(4.64,0.000007)(4.66,0.000007)(4.68,0.000007)(4.7,0.000006)(4.72,0.000006)(4.74,0.000006)(4.76,0.000006)(4.78,0.000006)(4.8,0.000006)(4.82,0.000005)(4.84,0.000005)(4.86,0.000005)(4.88,0.000005)(4.9,0.000005)
      };
      \legend{ {Empirical eigenvalue distribution},{Limiting density $\hat{p}_\rho$} }
    \end{axis}
  \end{tikzpicture}
  \caption{Histogram of the eigenvalues of $\hat{C}_N$ (Abramovich--Pascal type) for $n=2048$, $N=256$, $C_N=\frac13\diag(I_{128},5I_{128})$, $\rho=0.2$, versus limiting eigenvalue distribution.}
  \label{fig:hist_hatC}
\end{figure}
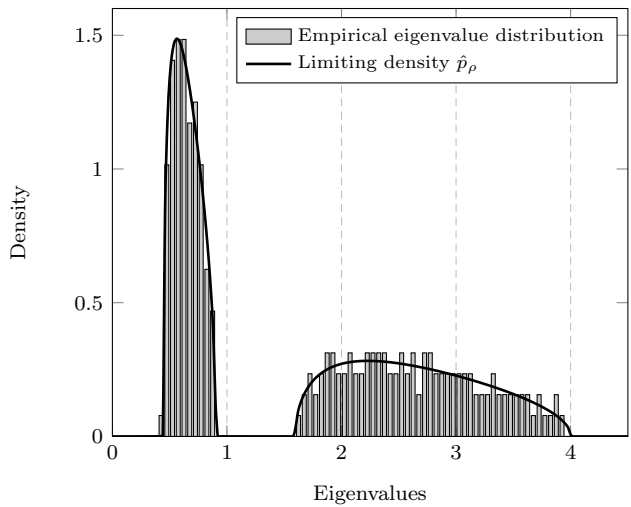

\begin{figure}[h!]
  \centering
  \begin{tikzpicture}[font=\footnotesize]
    \renewcommand{\axisdefaulttryminticks}{4} 
    \tikzstyle{every major grid}+=[style=densely dashed]       
    \tikzstyle{every axis y label}+=[yshift=-10pt] 
    \tikzstyle{every axis x label}+=[yshift=5pt]
    \tikzstyle{every axis legend}+=[cells={anchor=west},fill=white,
        at={(0.98,0.98)}, anchor=north east, font=\scriptsize ]
    \begin{axis}[
      xmin=0,
      ymin=0,
      xmax=2.5,
      ymax=3,
      bar width=1.3pt,
      grid=major,
      ymajorgrids=false,
      scaled ticks=true,
      xlabel={Eigenvalues},
      ylabel={Density}
      ]
      \addplot+[ybar,mark=none,color=black,fill=gray!40!white,area legend] coordinates{
(0.012500,0.000000)(0.037500,0.000000)(0.062500,0.000000)(0.087500,0.000000)(0.112500,0.000000)(0.137500,0.000000)(0.162500,0.000000)(0.187500,0.000000)(0.212500,0.000000)(0.237500,0.000000)(0.262500,0.000000)(0.287500,0.000000)(0.312500,0.156250)(0.337500,1.875000)(0.362500,2.500000)(0.387500,2.812500)(0.412500,2.343750)(0.437500,2.500000)(0.462500,2.031250)(0.487500,2.031250)(0.512500,1.562500)(0.537500,1.406250)(0.562500,0.781250)(0.587500,0.000000)(0.612500,0.000000)(0.637500,0.000000)(0.662500,0.000000)(0.687500,0.000000)(0.712500,0.000000)(0.737500,0.000000)(0.762500,0.000000)(0.787500,0.000000)(0.812500,0.000000)(0.837500,0.000000)(0.862500,0.000000)(0.887500,0.000000)(0.912500,0.000000)(0.937500,0.000000)(0.962500,0.000000)(0.987500,0.312500)(1.012500,0.156250)(1.037500,0.312500)(1.062500,0.312500)(1.087500,0.312500)(1.112500,0.468750)(1.137500,0.468750)(1.162500,0.625000)(1.187500,0.312500)(1.212500,0.625000)(1.237500,0.312500)(1.262500,0.625000)(1.287500,0.468750)(1.312500,0.468750)(1.337500,0.625000)(1.362500,0.312500)(1.387500,0.625000)(1.412500,0.468750)(1.437500,0.468750)(1.462500,0.468750)(1.487500,0.625000)(1.512500,0.468750)(1.537500,0.312500)(1.562500,0.625000)(1.587500,0.312500)(1.612500,0.468750)(1.637500,0.625000)(1.662500,0.468750)(1.687500,0.156250)(1.712500,0.625000)(1.737500,0.468750)(1.762500,0.468750)(1.787500,0.312500)(1.812500,0.312500)(1.837500,0.312500)(1.862500,0.312500)(1.887500,0.312500)(1.912500,0.468750)(1.937500,0.312500)(1.962500,0.312500)(1.987500,0.312500)(2.012500,0.312500)(2.037500,0.312500)(2.062500,0.156250)(2.087500,0.312500)(2.112500,0.156250)(2.137500,0.312500)(2.162500,0.156250)(2.187500,0.312500)(2.212500,0.000000)(2.237500,0.312500)(2.262500,0.156250)(2.287500,0.000000)(2.312500,0.156250)(2.337500,0.000000)(2.362500,0.000000)(2.387500,0.000000)(2.412500,0.000000)(2.437500,0.000000)(2.462500,0.000000)(2.487500,0.000000)(2.512500,0.000000)(2.537500,0.000000)(2.562500,0.000000)(2.587500,0.000000)(2.612500,0.000000)(2.637500,0.000000)(2.662500,0.000000)(2.687500,0.000000)(2.712500,0.000000)(2.737500,0.000000)(2.762500,0.000000)(2.787500,0.000000)(2.812500,0.000000)(2.837500,0.000000)(2.862500,0.000000)(2.887500,0.000000)(2.912500,0.000000)
      };
      \addplot[black,smooth,line width=1pt] plot coordinates{
	      (0.,0.000098)(0.02,0.000108)(0.04,0.000120)(0.06,0.000133)(0.08,0.000150)(0.1,0.000170)(0.12,0.000194)(0.14,0.000225)(0.16,0.000264)(0.18,0.000315)(0.2,0.000385)(0.22,0.000484)(0.24,0.000634)(0.26,0.000882)(0.28,0.001360)(0.3,0.002641)(0.32,0.184157)(0.34,2.075148)(0.36,2.492919)(0.38,2.618434)(0.4,2.606305)(0.42,2.516640)(0.44,2.380210)(0.46,2.208488)(0.48,2.011036)(0.5,1.786991)(0.52,1.531941)(0.54,1.234581)(0.56,0.852742)(0.58,0.010712)(0.6,0.001422)(0.62,0.000845)(0.64,0.000600)(0.66,0.000462)(0.68,0.000375)(0.7,0.000315)(0.72,0.000273)(0.74,0.000242)(0.76,0.000219)(0.78,0.000202)(0.8,0.000191)(0.82,0.000184)(0.84,0.000181)(0.86,0.000183)(0.88,0.000191)(0.9,0.000208)(0.92,0.000241)(0.94,0.000310)(0.96,0.000534)(0.98,0.125224)(1.,0.228642)(1.02,0.290055)(1.04,0.334227)(1.06,0.368299)(1.08,0.394998)(1.1,0.416936)(1.12,0.434857)(1.14,0.449752)(1.16,0.461684)(1.18,0.471263)(1.2,0.479192)(1.22,0.485436)(1.24,0.490219)(1.26,0.493981)(1.28,0.496372)(1.3,0.498203)(1.32,0.498888)(1.34,0.498755)(1.36,0.498275)(1.38,0.496803)(1.4,0.495030)(1.42,0.492682)(1.44,0.489849)(1.46,0.486661)(1.48,0.483027)(1.5,0.479064)(1.52,0.474818)(1.54,0.470219)(1.56,0.465317)(1.58,0.460379)(1.6,0.454876)(1.62,0.449154)(1.64,0.443409)(1.66,0.437309)(1.68,0.430919)(1.7,0.424579)(1.72,0.417800)(1.74,0.410912)(1.76,0.403885)(1.78,0.396688)(1.8,0.389460)(1.82,0.381847)(1.84,0.374051)(1.86,0.366057)(1.88,0.357780)(1.9,0.349453)(1.92,0.341144)(1.94,0.332151)(1.96,0.323424)(1.98,0.314126)(2.,0.304808)(2.02,0.295001)(2.04,0.285033)(2.06,0.274657)(2.08,0.263864)(2.1,0.252983)(2.12,0.241590)(2.14,0.229618)(2.16,0.216767)(2.18,0.203619)(2.2,0.189758)(2.22,0.174331)(2.24,0.158087)(2.26,0.139639)(2.28,0.118851)(2.3,0.093790)(2.32,0.059365)(2.34,0.000813)(2.36,0.000137)(2.38,0.000097)(2.4,0.000077)(2.42,0.000065)(2.44,0.000056)(2.46,0.000050)(2.48,0.000045)(2.5,0.000041)(2.52,0.000037)(2.54,0.000034)(2.56,0.000032)(2.58,0.000030)(2.6,0.000028)(2.62,0.000026)(2.64,0.000025)(2.66,0.000024)(2.68,0.000022)(2.7,0.000021)(2.72,0.000020)(2.74,0.000019)(2.76,0.000019)(2.78,0.000018)(2.8,0.000017)(2.82,0.000016)(2.84,0.000016)(2.86,0.000015)(2.88,0.000015)(2.9,0.000014)
      };
      \legend{ {Empirical eigenvalue distribution},{Limiting density $\check{p}_\rho$} }
    \end{axis}
  \end{tikzpicture}
  \caption{Histogram of the eigenvalues of $\check{C}_N$ (Chen type) for $n=2048$, $N=256$, $C_N=\frac13\diag(I_{128},5I_{128})$, $\rho=0.2$, versus limiting eigenvalue distribution.}
  \label{fig:hist_checkC}
\end{figure}
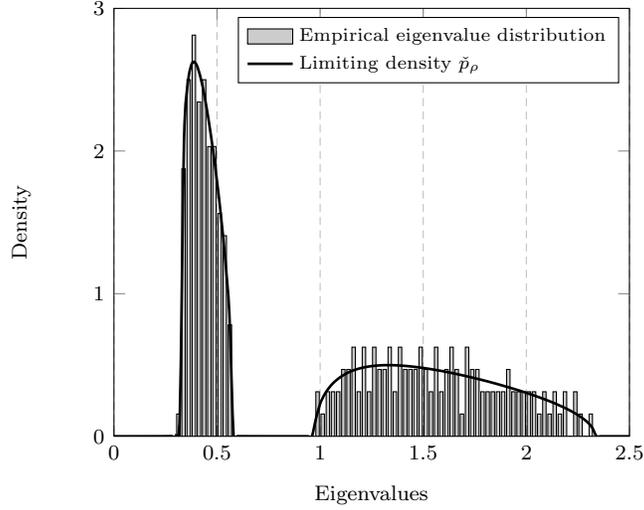

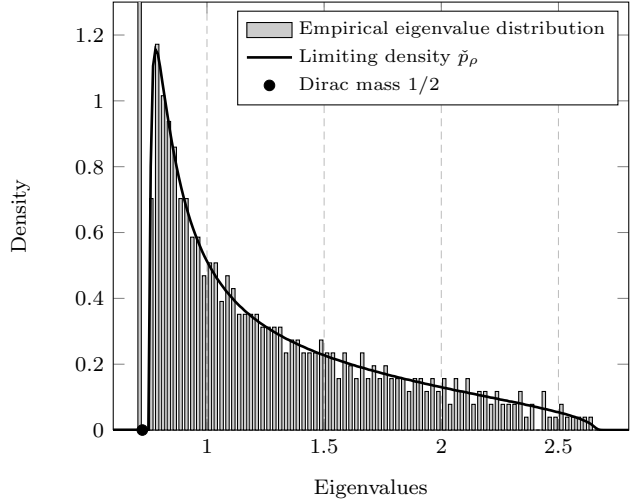
\begin{figure}[h!]
  \centering
  \begin{tikzpicture}[font=\footnotesize]
    \renewcommand{\axisdefaulttryminticks}{4} 
    \tikzstyle{every major grid}+=[style=densely dashed]       
    \tikzstyle{every axis y label}+=[yshift=-10pt] 
    \tikzstyle{every axis x label}+=[yshift=5pt]
    \tikzstyle{every axis legend}+=[cells={anchor=west},fill=white,
        at={(0.98,0.98)}, anchor=north east, font=\scriptsize ]
    \begin{axis}[
      xmin=0.6,
      ymin=0,
      xmax=2.8,
      ymax=1.3,
      bar width=1.3pt,
      grid=major,
      ymajorgrids=false,
      scaled ticks=true,
      xlabel={Eigenvalues},
      ylabel={Density}
      ]
      \addplot+[ybar,mark=none,color=black,fill=gray!40!white,area legend] coordinates{
	      (0.0125,0.000000)(0.0375,0.000000)(0.0625,0.000000)(0.0875,0.000000)(0.1125,0.000000)(0.1375,0.000000)(0.1625,0.000000)(0.1875,0.000000)(0.2125,0.000000)(0.2375,0.000000)(0.2625,0.000000)(0.2875,0.000000)(0.3125,0.000000)(0.3375,0.000000)(0.3625,0.000000)(0.3875,0.000000)(0.4125,0.000000)(0.4375,0.000000)(0.4625,0.000000)(0.4875,0.000000)(0.5125,0.000000)(0.5375,0.000000)(0.5625,0.000000)(0.5875,0.000000)(0.6125,0.000000)(0.6375,0.000000)(0.6625,0.000000)(0.6875,0.000000)(0.7125,20.000000)(0.7375,0.000000)(0.7625,0.703125)(0.7875,1.171875)(0.8125,1.015625)(0.8375,0.937500)(0.8625,0.859375)(0.8875,0.703125)(0.9125,0.703125)(0.9375,0.585938)(0.9625,0.585938)(0.9875,0.468750)(1.0125,0.507812)(1.0375,0.507812)(1.0625,0.390625)(1.0875,0.468750)(1.1125,0.429688)(1.1375,0.351562)(1.1625,0.351562)(1.1875,0.351562)(1.2125,0.351562)(1.2375,0.312500)(1.2625,0.312500)(1.2875,0.312500)(1.3125,0.312500)(1.3375,0.234375)(1.3625,0.273438)(1.3875,0.273438)(1.4125,0.234375)(1.4375,0.234375)(1.4625,0.234375)(1.4875,0.273438)(1.5125,0.234375)(1.5375,0.234375)(1.5625,0.156250)(1.5875,0.234375)(1.6125,0.195312)(1.6375,0.156250)(1.6625,0.234375)(1.6875,0.156250)(1.7125,0.195312)(1.7375,0.156250)(1.7625,0.195312)(1.7875,0.156250)(1.8125,0.156250)(1.8375,0.156250)(1.8625,0.117188)(1.8875,0.156250)(1.9125,0.156250)(1.9375,0.117188)(1.9625,0.156250)(1.9875,0.117188)(2.0125,0.156250)(2.0375,0.078125)(2.0625,0.156250)(2.0875,0.117188)(2.1125,0.156250)(2.1375,0.078125)(2.1625,0.117188)(2.1875,0.117188)(2.2125,0.078125)(2.2375,0.117188)(2.2625,0.078125)(2.2875,0.078125)(2.3125,0.078125)(2.3375,0.117188)(2.3625,0.039062)(2.3875,0.078125)(2.4125,0.000000)(2.4375,0.117188)(2.4625,0.039062)(2.4875,0.039062)(2.5125,0.078125)(2.5375,0.039062)(2.5625,0.039062)(2.5875,0.039062)(2.6125,0.039062)(2.6375,0.039062)(2.6625,0.000000)(2.6875,0.000000)(2.7125,0.000000)(2.7375,0.000000)(2.7625,0.000000)(2.7875,0.000000)(2.8125,0.000000)(2.8375,0.000000)(2.8625,0.000000)(2.8875,0.000000)(2.9125,0.000000)(2.9375,0.000000)(2.9625,0.000000)(2.9875,0.000000)(3.0125,0.000000)(3.0375,0.000000)(3.0625,0.000000)(3.0875,0.000000)(3.1125,0.000000)(3.1375,0.000000)(3.1625,0.000000)(3.1875,0.000000)(3.2125,0.000000)(3.2375,0.000000)(3.2625,0.000000)(3.2875,0.000000)
      };
      \addplot[black,line width=1pt] plot coordinates{
	      (0.,0.000044)(0.75,0.00086)(0.76,0.802005)(0.77,1.105944)(0.78,1.155881)(0.79,1.143222)(0.8,1.103494)(0.81,1.054911)(0.82,1.008544)(0.83,0.962583)(0.84,0.917172)(0.85,0.875408)(0.86,0.836815)(0.87,0.801003)(0.88,0.767822)(0.89,0.736846)(0.9,0.708273)(0.91,0.681838)(0.92,0.657568)(0.93,0.634491)(0.94,0.613565)(0.95,0.594140)(0.96,0.575925)(0.97,0.558761)(0.98,0.541929)(0.99,0.526678)(1.,0.512789)(1.01,0.499662)(1.02,0.487175)(1.03,0.475235)(1.04,0.463178)(1.05,0.453267)(1.06,0.443045)(1.07,0.433495)(1.08,0.423618)(1.09,0.415546)(1.1,0.407080)(1.11,0.399199)(1.12,0.390850)(1.13,0.383566)(1.14,0.376399)(1.15,0.370373)(1.16,0.363159)(1.17,0.356890)(1.18,0.351585)(1.19,0.345065)(1.2,0.339433)(1.21,0.334727)(1.22,0.328755)(1.23,0.323750)(1.24,0.319158)(1.25,0.314025)(1.26,0.309659)(1.27,0.305251)(1.28,0.300588)(1.29,0.296641)(1.3,0.292187)(1.31,0.288399)(1.32,0.284808)(1.33,0.280285)(1.34,0.276547)(1.35,0.273605)(1.36,0.269516)(1.37,0.266482)(1.38,0.263017)(1.39,0.259442)(1.4,0.256118)(1.41,0.253437)(1.42,0.250201)(1.43,0.246605)(1.44,0.244041)(1.45,0.241169)(1.46,0.237913)(1.47,0.235087)(1.48,0.232227)(1.49,0.229727)(1.5,0.227468)(1.51,0.224380)(1.52,0.221738)(1.53,0.219117)(1.54,0.216652)(1.55,0.214656)(1.56,0.212431)(1.57,0.209883)(1.58,0.207579)(1.59,0.205396)(1.6,0.202862)(1.61,0.200588)(1.62,0.198183)(1.63,0.196472)(1.64,0.193702)(1.65,0.191859)(1.66,0.190038)(1.67,0.188088)(1.68,0.185943)(1.69,0.183376)(1.7,0.182022)(1.71,0.180101)(1.72,0.177791)(1.73,0.176217)(1.74,0.174006)(1.75,0.171846)(1.76,0.170589)(1.77,0.168069)(1.78,0.166423)(1.79,0.164441)(1.8,0.163304)(1.81,0.161117)(1.82,0.159263)(1.83,0.157998)(1.84,0.155676)(1.85,0.153970)(1.86,0.152922)(1.87,0.151279)(1.88,0.149022)(1.89,0.147277)(1.9,0.146309)(1.91,0.144569)(1.92,0.142886)(1.93,0.140828)(1.94,0.139493)(1.95,0.137950)(1.96,0.136367)(1.97,0.134442)(1.98,0.133248)(1.99,0.131764)(2.,0.130347)(2.01,0.128307)(2.02,0.127171)(2.03,0.125676)(2.04,0.124310)(2.05,0.122216)(2.06,0.120723)(2.07,0.119317)(2.08,0.117855)(2.09,0.116975)(2.1,0.115496)(2.11,0.114050)(2.12,0.111899)(2.13,0.110904)(2.14,0.109655)(2.15,0.107994)(2.16,0.106007)(2.17,0.105015)(2.18,0.103320)(2.19,0.102103)(2.2,0.100662)(2.21,0.098675)(2.22,0.097874)(2.23,0.096339)(2.24,0.095117)(2.25,0.092872)(2.26,0.091723)(2.27,0.090255)(2.28,0.089299)(2.29,0.086958)(2.3,0.086372)(2.31,0.084776)(2.32,0.083421)(2.33,0.081766)(2.34,0.080441)(2.35,0.078938)(2.36,0.076677)(2.37,0.075427)(2.38,0.074217)(2.39,0.072799)(2.4,0.070655)(2.41,0.069295)(2.42,0.067368)(2.43,0.065737)(2.44,0.064076)(2.45,0.062894)(2.46,0.060856)(2.47,0.059282)(2.48,0.057244)(2.49,0.055301)(2.5,0.054208)(2.51,0.051773)(2.52,0.050303)(2.53,0.048544)(2.54,0.046171)(2.55,0.044125)(2.56,0.041608)(2.57,0.039741)(2.58,0.036803)(2.59,0.034725)(2.6,0.031561)(2.61,0.028610)(2.62,0.025471)(2.63,0.021917)(2.64,0.017521)(2.65,0.011755)(2.66,0.005971)(2.67,0.001816)(2.68,0.000375)(2.69,0.000097)(2.7,0.000049)(2.71,0.000035)(2.72,0.000030)(2.73,0.000028)(2.74,0.000026)(2.75,0.000024)(2.76,0.000023)(2.77,0.000022)(2.78,0.000021)(2.79,0.000020)(2.8,0.000019)(2.81,0.000019)(2.82,0.000018)(2.83,0.000017)(2.84,0.000017)(2.85,0.000016)(2.86,0.000016)(2.87,0.000015)(2.88,0.000015)(2.89,0.000015)(2.9,0.000014)(2.91,0.000014)(2.92,0.000014)(2.93,0.000013)(2.94,0.000013)(2.95,0.000013)(2.96,0.000013)(2.97,0.000012)(2.98,0.000012)(2.99,0.000012)(3.,0.000012)(3.01,0.000011)(3.02,0.000011)(3.03,0.000011)(3.04,0.000011)(3.05,0.000011)(3.06,0.000010)(3.07,0.000010)(3.08,0.000010)(3.09,0.000010)(3.1,0.000010)(3.11,0.000010)(3.12,0.000010)(3.13,0.000009)(3.14,0.000009)(3.15,0.000009)(3.16,0.000009)(3.17,0.000009)(3.18,0.000009)(3.19,0.000009)(3.2,0.000009)(3.21,0.000008)(3.22,0.000008)(3.23,0.000008)(3.24,0.000008)(3.25,0.000008)(3.26,0.000008)(3.27,0.000008)(3.28,0.000008)(3.29,0.000008)
      };
      \addplot[black,line width=0.2pt,mark=*,only marks] plot coordinates{
      		(0.7242,0)
      };
      \legend{ {Empirical eigenvalue distribution},{Limiting density $\check{p}_\rho$},{Dirac mass $1/2$}}
    \end{axis}
  \end{tikzpicture}
  \caption{Histogram of the eigenvalues of $\check{C}_N$ (Chen type) for $n=512$, $N=1024$, $C_N=\frac13\diag(I_{128},5I_{128})$, $\rho=0.8$, versus limiting eigenvalue distribution.}
  \label{fig:hist_checkC_N>n}
\end{figure}

From Corollary~\ref{cor:limit}, $\hat{\mu}_\rho$ is continuous on $(\rho,\infty)$ so that $\hat{\mu}_\rho(dx)=\hat{p}_\rho(x)dx$ where, from the inverse Stieltjes transform formula (see e.g., \citep{SIL06}) for all $x\in (\rho,\infty)$,
\begin{align*}
	\hat{p}_\rho(x) &= \lim_{\varepsilon\to 0}\frac1\pi \Im\left[ m_{\hat{\mu}_\rho}(x+\imath \varepsilon)\right].
\end{align*}
Letting $\varepsilon>0$ small and approximating $\hat{p}_\rho(x)$ by $\frac1\pi \Im[ m_{\hat{\mu}_\rho}(x+\imath \varepsilon)]$ allows one to depict $\hat{p}_\rho$ approximately. Similarly, $\check{\mu}_\rho(dx)=\check{p}_\rho(x)dx$ for all $x\in (T_\rho(1-\rho+T_\rho)^{-1},\infty)$ which can be obtained equivalently. This is performed in Figure~\ref{fig:hist_hatC} and Figure~\ref{fig:hist_checkC} which depict the histogram of the eigenvalues of $\hat{C}_N(\rho)$ and $\check{C}_N(\rho)$ for $\rho=0.2$, $N=256$, $n=2048$, $C_N=\diag(I_{128},5I_{128})$, versus their limiting distributions for $c=1/8$. Figure~\ref{fig:hist_checkC_N>n} depicts $\check{C}_N(\rho)$ for $\rho=0.8$, $N=1024$, $n=512$, $C_N=\diag(I_{128},5I_{128})$ versus its limiting distribution for $c=2$. Note that, when $c=1/8$, the eigenvalues of $\check{C}_N(\rho)$ concentrate in two bulks close to $1/3$ and $5/3$, as expected. Due to the different trace normalization of $\hat{C}_N(\rho)$, the same reasoning holds up to a multiplicative constant. However, when $c=2$, the eigenvalues of $\check{C}_N(\rho)$ are quite remote from masses in $1/3$ and $5/3$, an observation known since \citep{MAR67}.

Another corollary of Theorem~\ref{th:Chitour} and Theorem~\ref{th:Wiesel} is the joint convergence (over both $\rho$ and the eigenvalue index) of the individual eigenvalues of $\hat{C}_N(\rho)$ to those of $\hat{S}_N(\rho)$ and of the individual eigenvalues of $\check{C}_N(\rho)$ to those of $\check{S}_N(\rho)$, as well as the joint convergence over $\rho$ of the moments of the empirical spectral distributions of $\hat{C}_N(\rho)$ and $\check{C}_N(\rho)$. These joint convergence properties are fundamental in problems of optimization of the parameter $\rho$ as discussed in Section~\ref{sec:shrink}.

\begin{corollary}[Joint convergence properties]
	\label{cor:joint}
	Under the settings of Theorem~\ref{th:Chitour} and Theorem~\ref{th:Wiesel}, 
	\begin{align*}
		\sup_{\rho\in\hat{\mathcal R}_\varepsilon}\max_{1\leq i\leq n} \left| \lambda_i(\hat{C}_N(\rho)) - \lambda_i(\hat{S}_N(\rho)) \right| &\asto 0 \\
		\sup_{\rho\in\check{\mathcal R}_\varepsilon}\max_{1\leq i\leq n} \left| \lambda_i(\check{C}_N(\rho)) - \lambda_i(\check{S}_N(\rho)) \right| &\asto 0.
	\end{align*}
	This result implies 
	\begin{align*}
		\limsup_N \sup_{\rho\in\hat{\mathcal R}_\varepsilon}\Vert \hat{C}_N(\rho)\Vert &<\infty \\
		\limsup_N \sup_{\rho\in\check{\mathcal R}_\varepsilon} \Vert \check{C}_N(\rho)\Vert &<\infty.
	\end{align*}
	almost surely. This, and the weak convergence of Corollary~\ref{cor:limit}, in turn induce that, for each $\ell\in \NN$,
	\begin{align*}
		\sup_{\rho\in\hat{\mathcal R}_\varepsilon} \left|\frac1N\tr \left(\hat{C}_N(\rho)^\ell\right) - M_{\hat{\mu}_\rho,\ell} \right| &\asto 0 \\
		\sup_{\rho\in\check{\mathcal R}_\varepsilon} \left|\frac1N\tr \left(\check{C}_N(\rho)^\ell\right) - M_{\check{\mu}_\rho,\ell} \right| &\asto 0
	\end{align*}
where we recall that $M_{\mu,\ell}=\int t^\ell \mu(dt)\in(0,\infty]$ for any probability measure $\mu$ with support in $\RR^+$; in particular, $M_{\hat{\mu}_\rho,1}=\frac1{\hat\gamma(\rho)}\frac{1-\rho}{1-(1-\rho)c}+\rho$ and $M_{\check{\mu}_\rho,1}=1$.
\end{corollary}
\begin{proof}
	The proof is provided in Section~\ref{sec:proof_joint}.
\end{proof}

\section{Application to optimal shrinkage}
\label{sec:shrink}

We now apply Theorems~\ref{th:Chitour} and \ref{th:Wiesel} to the problem of optimal linear shrinkage, originally considered in \citep{LED04} for the simpler sample covariance matrix model. The optimal linear shrinkage problem consists in choosing $\rho$ to be such that a certain distance metric between $\hat{C}_N(\rho)$ (or $\check{C}_N(\rho)$) and $C_N$ is minimized, therefore allowing for a more appropriate estimation of $C_N$ via $\hat{C}_N(\rho)$ or $\check{C}_N(\rho)$.
The metric selected here is the squared Frobenius norm of the difference between the (possibly scaled) robust estimators and $C_N$, which has the advantage of being a widespread matrix distance (e.g., as considered in \citep{LED04}) and a metric amenable to mathematical analysis.\footnote{Alternative metrics (such as the geodesic distance on the cone of nonnegative definite matrices) can be similarly considered. The appropriate choice of such a metric heavily depends on the ultimate problem to optimize.}
In \citep{CHE11}, the authors studied this problem in the specific case of $\check{C}_N(\rho)$ but did not find an expression for the optimal theoretical $\rho$ due to the involved structure of $\check{C}_N(\rho)$ for all finite $N,n$ and therefore resorted to solving an approximate problem, the solution of which is denoted here $\check\rho_O$. Instead, we show that for large $N,n$ values the optimal $\rho$ under study converges to a limiting value $\check\rho^\star$ that takes an extremely simple explicit expression and a similar result holds for $\hat{C}_N(\rho)$ for which an equivalent optimal $\hat{\rho}^\star$ is defined.

Our first result is a lemma of fundamental importance which demonstrates that, up to a change in the variable $\rho$, $\hat{S}_N(\rho)/M_{\hat\mu_{\rho},1}$ and $\check{S}_N(\rho)$ (constructed from the samples $x_1,\ldots,x_n$) are completely equivalent to the original Ledoit--Wolf linear shrinkage model for the (non observable) samples $z_1,\ldots,z_n$.

\begin{lemma}[Model Equivalence]
	\label{lem:equivalent}
For each $\rho\in (0,1]$, there exist unique $\hat\rho\in (\max\{0,1-c^{-1}\},1]$ and $\check\rho\in (0,1]$ such that
\begin{align*}
	\frac{\hat{S}_N(\hat\rho)}{M_{\hat\mu_{\hat\rho},1}}=\check{S}_N(\check\rho) &= (1-\rho) \frac1n\sum_{i=1}^n z_iz_i^* + \rho I_N.
\end{align*}
Besides, the maps $(0,1]\to (\max\{0,1-c^{-1}\},1]$, $\rho\mapsto \hat\rho$ and $(0,1]\to (0,1]$, $\rho\mapsto \check\rho$ thus defined are continuously increasing and onto.
\end{lemma}
\begin{proof}
	The proof is provided in Section~\ref{sec:proof_lemma}.
\end{proof}

Along with Theorem~\ref{th:Chitour} and Theorem~\ref{th:Wiesel}, Lemma~\ref{lem:equivalent} indicates that, up to a change in the variable $\rho$, $\hat{C}_N(\rho)$ and $\check{C}_N(\rho)$ can be somewhat viewed as asymptotically equivalent (but there is no saying whether they can be claimed equivalent for all finite $N,n$).
As such, thanks to Lemma~\ref{lem:equivalent}, we now show that the optimal shrinkage parameters $\rho$ for both $\hat{C}_N(\rho)/(\frac1N\tr \hat{C}_N(\rho))$ and $\check{C}_N(\rho)$ lead to the same asymptotic performance, which corresponds to the asymptotically optimal Ledoit--Wolf linear shrinkage performance but for the vectors $z_1,\ldots,z_n$.

\begin{proposition}[Optimal Shrinkage]
	\label{prop:shrink}
For each $\rho\in (0,1]$, define\footnote{Recall that, for $A$ Hermitian, $\frac1N\tr (A^2)=\frac1N\tr(AA^*)=\frac1N\Vert A\Vert_F^2$ with $\Vert\cdot\Vert_F$ the Frobenius norm for matrices.}
\begin{align*}
	\hat{D}_N(\rho) &= \frac1N\tr\left(\left( \frac{\hat{C}_N(\rho)}{\frac1N\tr \hat{C}_N(\rho)} - C_N \right)^2\right) \\
	\check{D}_N(\rho) &= \frac1N\tr\left(\left(  \check{C}_N(\rho) - C_N \right)^2\right).
\end{align*}
Also denote $D^\star=c \frac{M_{\nu,2}-1}{c+M_{\nu,2}-1}$, $\rho^\star=\frac{c}{c+M_{\nu,2}-1}$, and $\hat\rho^\star\in(\max\{0,1-c^{-1}\},1]$, $\check{\rho}^\star\in(0,1]$ the unique solutions to
\begin{align*}
	\frac{\hat\rho^\star}{\frac1{\hat\gamma(\hat\rho^\star)}\frac{1-\hat\rho^\star}{1-(1-\hat\rho^\star)c}+\hat\rho^\star} &= \frac{T_{\check\rho^\star}}{1-\check\rho^\star+T_{\check\rho^\star}} = \rho^\star.
\end{align*}
Then, letting $\varepsilon<\min(\hat\rho^\star-\max\{0,1-c^{-1}\},\check\rho^\star)$, under the setting of Theorem~\ref{th:Chitour} and Theorem~\ref{th:Wiesel}, 
\begin{align*}
	\inf_{\rho \in \hat{\mathcal R}_\varepsilon} \hat{D}_N(\rho) &\asto D^\star,\quad \inf_{\rho \in \check{\mathcal R}_\varepsilon} \check{D}_N(\rho) \asto D^\star
\end{align*}
and
\begin{align*}
	\hat{D}_N(\hat\rho^\star) &\asto D^\star,\quad \check{D}_N(\check\rho^\star) \asto D^\star.
\end{align*}
Moreover, letting $\hat\rho_N$ and $\check\rho_N$ be random variables such that $\hat\rho_N\asto \hat\rho^\star$ and $\check\rho_N\asto \check\rho^\star$, 
\begin{align*}
	\hat{D}_N(\hat\rho_N) &\asto D^\star, \quad \check{D}_N(\check\rho_N) \asto D^\star.
\end{align*}
\end{proposition}
\begin{proof}
	The proof is provided in Section~\ref{sec:proof_optimal}.
\end{proof}

The last part of Proposition~\ref{prop:shrink} states that, if consistent estimates $\hat\rho_N$ and $\check\rho_N$ of $\hat\rho^\star$ and $\check{\rho}^\star$ exist, then they have optimal shrinkage performance in the large $N,n$ limit. Such estimates may of course be defined in multiple ways. We present below a simple example based on $\hat{C}_N(\rho)$ and $\check{C}_N(\rho)$.

\begin{proposition}[Optimal Shrinkage Estimate]
	\label{prop:optimal_shrink}
	Under the setting of Proposition~\ref{prop:shrink}, let $\hat\rho_N\in (\max\{0,1-c^{-1}\},1]$ and $\check\rho_N\in (0,1]$ be solutions (not necessarily unique) to 
\begin{align*}
	\frac{\hat{\rho}_N}{\frac1N\tr \hat{C}_N(\hat{\rho}_N)} &= \frac{c_N}{\frac1N\tr \left[\left( \frac1n\sum_{i=1}^n \frac{x_ix_i^*}{\frac1N\Vert x_i\Vert^2} \right)^2\right]-1}  \\
	\frac{\check{\rho}_N \frac1n \sum_{i=1}^n \frac{x_i^*\check{C}_N(\check{\rho}_N)^{-1}x_i}{\Vert x_i\Vert^2}}{1-\check{\rho}_N + \check{\rho}_N \frac1n \sum_{i=1}^n \frac{x_i^*\check{C}_N(\check{\rho}_N)^{-1}x_i}{\Vert x_i\Vert^2}}  &= \frac{c_N}{\frac1N\tr\left[ \left( \frac1n\sum_{i=1}^n \frac{x_ix_i^*}{\frac1N\Vert x_i\Vert^2} \right)^2\right]-1} \nonumber
\end{align*}
defined arbitrarily when no such solutions exist. Then $\hat{\rho}_N\asto \hat\rho^\star$ and $\check{\rho}_N\asto \check\rho^\star$, so that $\hat{D}_N(\hat{\rho}_N)\asto D^\star$ and $\check{D}_N(\check{\rho}_N)\asto D^\star$.
\end{proposition}
\begin{proof}
	The proof is deferred to Section~\ref{sec:proof_optimal_shrink}.
\end{proof}

Figure~\ref{fig:shrinkage} illustrates the performance in terms of the metric $\check{D}_N$ of the empirical shrinkage coefficient $\check{\rho}_N$ introduced in Proposition~\ref{prop:optimal_shrink} versus the optimal value $\inf_{\rho\in (0,1]}\{\check{D}_N(\rho)\}$, averaged over $10\,000$ Monte Carlo simulations. We also present in this graph the almost sure limiting value $D^\star$ of both $\check{D}_N(\check{\rho}_N)$ and $\inf_{\rho\in \check{\mathcal R}_\varepsilon}\{\check{D}_N(\rho)\}$ for some sufficiently small $\varepsilon$, as well as $\check{D}_N(\check{\rho}_O)$ of $\check{\rho}_O$ defined in \citep[Equation~(12)]{CHE11} as the minimizing solution of $\EE[\frac1N\tr (\check{C}_{O}(\rho)-C_N)^2]$ with $\check{C}_{O}(\rho)$ the so-called ``clairvoyant estimator''
\begin{align*}
	\check{C}_{O}(\rho) &= (1-\rho)\frac1n\sum_{i=1}^n \frac{x_ix_i^*}{\frac1Nx_i^*C_N^{-1}x_i} + \rho I_N.
\end{align*}
We consider in this graph $N=32$ constant, $n\in\{2^k,k=1,\ldots,7\}$, and $C_N=[C_N]_{i,j=1}^N$ with $[C_N]_{ij}=r^{|i-j|}$, $r=0.7$, which is the same setting as considered in \citep[Section~4]{CHE11}. 

It appears in Figure~\ref{fig:shrinkage} that a significant improvement is brought by $\check{\rho}_N$ over $\check{\rho}_O$, especially for small $n$, which translates the poor quality of $\check{C}_{O}(\rho)$ as an approximation of $\check{C}_N(\rho)$ for large values of $c_N$ (obviously linked to $\frac1Nx_i^*C_N^{-1}x_i$ being then a bad approximation for $\frac1Nx_i^*\check{C}_N(\rho)^{-1}x_i$). Another important remark is that, even for so small values of $N,n$, $\inf_{\rho\in (0,1]}\check{D}_N(\rho)$ is extremely close to the limiting optimal, suggesting here that the limiting results of Proposition~\ref{prop:shrink} are already met for small practical values. The approximation $\check{\rho}_N$ of $\check\rho^\star$, translated here through $\check{D}_N(\check{\rho}_N)$, also demonstrates good practical performance at small values of $N,n$.

We additionally mention that we produced similar curves for $\hat{C}_N(\rho)$ in place of $\check{C}_N(\rho)$ which happened to show virtually the same performance as the equivalents curves for $\check{C}_N(\rho)$. This is of course expected (with exact match) for $\inf_{\rho\in (0,1]}\hat{D}_N(\rho)$ which, up to the region $[0,\varepsilon)$, matches $\inf_{\rho\in (0,1]}\check{D}_N(\rho)$ for large enough $N,n$, and similarly for $\hat{D}_N(\hat{\rho}_N)$ since $\hat{\rho}_N$ was designed symmetrically to $\check{\rho}_N$.

Associated to Figure~\ref{fig:shrinkage} is Figure~\ref{fig:shrinkage_rho} which provides the shrinkage parameter values, optimal and approximated, for both the Abramovich--Pascal and Chen estimates, along with the clairvoyant $\check{\rho}_O$ of \citep{CHE11}. Recall that the $\hat{(\cdot)}$ values must only be compared to one another, and similarly for the $\check{(\cdot)}$ values (so in particular $\check{\rho}_O$ only compares against the $\check{(\cdot)}$ values). It appears here that $\check{\rho}_O$ is a rather poor estimate for $\argmin_{\rho\in (0,1]}\check{D}_N(\rho)$ for a large range of values of $n$. It tends in particular to systematically overestimate the weight to be put on the sample covariance matrix.

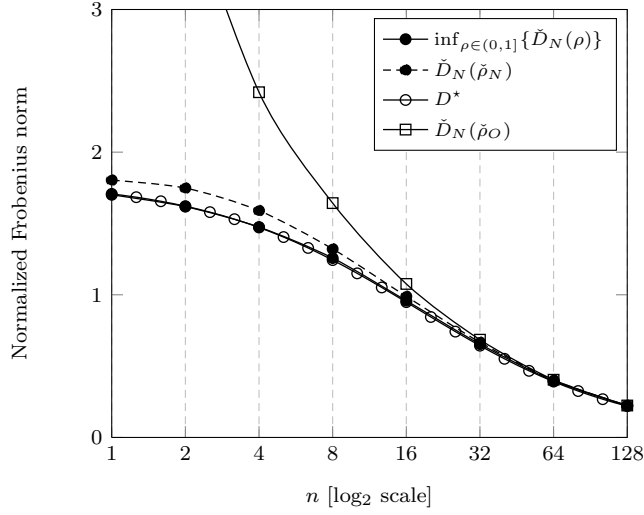
\begin{figure}[h!]
  \centering
  \begin{tikzpicture}[font=\footnotesize]
    \renewcommand{\axisdefaulttryminticks}{4} 
    \tikzstyle{every major grid}+=[style=densely dashed]       
    \tikzstyle{every axis y label}+=[yshift=-10pt] 
    \tikzstyle{every axis x label}+=[yshift=5pt]
    \tikzstyle{every axis legend}+=[cells={anchor=west},fill=white,
        at={(0.98,0.98)}, anchor=north east, font=\scriptsize ]
    \begin{semilogxaxis}[
      xmin=1,
      ymin=0,
      xmax=128,
      ymax=3,
      xtick={1,2,4,8,16,32,64,128},
      xticklabels={$1$,$2$,$4$,$8$,$16$,$32$,$64$,$128$},
      bar width=1.3pt,
      grid=major,
      ymajorgrids=false,
      scaled ticks=true,
      xlabel={$n$ [log$_2$ scale]},
      ylabel={Normalized Frobenius norm},
      log basis x={2}
      ]
      \addplot[black,mark=*,smooth,line width=0.5pt] plot coordinates{
	      (1,1.700861)(2,1.618079)(4,1.474713)(8,1.257387)(16,0.958598)(32,0.654461)(64,0.396244)(128,0.222277)
      };
      \addplot[black,densely dashed,smooth,mark=*,smooth,line width=0.5pt] plot coordinates{
	      (1,1.803825)(2,1.748016)(4,1.589927)(8,1.321316)(16,0.988920)(32,0.667180)(64,0.401279)(128,0.224243)
      };
      \addplot[domain=1:256,smooth,mark=o,line width=0.5pt] {32/x*(1.8038)/(32/x+1.8038)};
      \addplot[black,mark=square,smooth,line width=0.5pt] plot coordinates{
	      (1,6.061630)(2,3.885264)(4,2.420059)(8,1.642827)(16,1.076083)(32,0.685260)(64,0.404383)(128,0.223857)
      };
      \legend{{ $\inf_{\rho\in (0,1]} \{\check{D}_N(\rho)\}$ }, {$\check{D}_N(\check\rho_N)$}, {$D^\star$}, {$\check{D}_N(\check\rho_O)$} }
    \end{semilogxaxis}
  \end{tikzpicture}
  \caption{Performance of optimal shrinkage averaged over $10\,000$ Monte Carlo simulations, for $N=32$, various values of $n$, $[C_N]_{ij}=r^{|i-j|}$ with $r=0.7$; $\check\rho_N$ is given in Proposition~\ref{prop:optimal_shrink}; $\check\rho_O$ is the clairvoyant estimator proposed in \citep[Equation~(12)]{CHE11}; $D^\star$ taken with $c=N/n$.}
  \label{fig:shrinkage}
\end{figure}

\begin{figure}[h!]
  \centering
  \begin{tikzpicture}[font=\footnotesize]
    \renewcommand{\axisdefaulttryminticks}{4} 
    \tikzstyle{every major grid}+=[style=densely dashed]       
    \tikzstyle{every axis y label}+=[yshift=-10pt] 
    \tikzstyle{every axis x label}+=[yshift=5pt]
    \tikzstyle{every axis legend}+=[cells={anchor=west},fill=white,
        at={(0.02,0.02)}, anchor=south west, font=\scriptsize ]
    \begin{semilogxaxis}[
      xmin=1,
      ymin=0,
      xmax=128,
      ymax=1,
      xtick={1,2,4,8,16,32,64,128},
      xticklabels={$1$,$2$,$4$,$8$,$16$,$32$,$64$,$128$},
      bar width=1.3pt,
      grid=major,
      ymajorgrids=false,
      mark options={solid},
      scaled ticks=true,
      legend columns=3,
      xlabel={$n$ [log$_2$ scale]},
      ylabel={Shrinkage parameter},
      log basis x={2}
      ]
      \addplot[black,mark=triangle*,line width=0.5pt] plot coordinates{
(1,0.980687)(2,0.960952)(4,0.924197)(8,0.859237)(16,0.756050)(32,0.610550)(64,0.439700)(128,0.281400)
      };
      \addplot[black,densely dashed,mark=triangle*,smooth,line width=0.5pt] plot coordinates{
(1,1.000000)(2,0.986773)(4,0.951149)(8,0.881600)(16,0.779975)(32,0.639900)(64,0.473750)(128,0.310800)
      };
      \addplot[mark=triangle,line width=0.5pt] plot coordinates {
(1,0.988750)(2,0.967500)(4,0.935000)(8,0.870000)(16,0.770000)(32,0.620000)(64,0.460000)(128,0.300000)
      };
      \addplot[black,smooth,mark=*,line width=0.5pt] plot coordinates{
	      (1,0.980789)(2,0.958281)(4,0.914106)(8,0.829463)(16,0.682225)(32,0.475950)(64,0.274000)(128,0.139600)
      };
      \addplot[black,smooth,densely dashed,mark=*,smooth,line width=0.5pt] plot coordinates{
(1,1.000000)(2,0.985710)(4,0.946787)(8,0.862569)(16,0.719125)(32,0.514825)(64,0.307700)(128,0.158800)
      };
      \addplot[mark=o,smooth,line width=0.5pt] plot coordinates {
	      (1,0.980000)(2,0.960000)(4,0.920000)(8,0.840000)(16,0.690000)(32,0.490000)(64,0.290000)(128,0.150000)
      };
      \addplot[black,mark=square,smooth,line width=0.5pt] plot coordinates{
(1,0.947556)(2,0.900339)(4,0.818742)(8,0.693111)(16,0.530352)(32,0.360870)(64,0.220159)(128,0.123696)
      };
      \legend{  {$\hat\rho^\circ$}, {$\hat\rho_N$}, {$\hat\rho^\star$}, {$\check{\rho}^\circ$}, {$\check\rho_N$}, {$\check\rho^\star$}, {$\check\rho_O$}} 
    \end{semilogxaxis}
  \end{tikzpicture}
  \caption{Shrinkage parameter $\rho$ averaged over $10\,000$ Monte Carlo simulations, for $N=32$, various values of $n$, $[C_N]_{ij}=r^{|i-j|}$ with $r=0.7$; $\hat\rho_N$ and $\check\rho_N$ given in Proposition~\ref{prop:optimal_shrink}; $\check\rho_O$ is the clairvoyant estimator proposed in \citep[Equation~(12)]{CHE11}; $\rho^\star$, $\hat{\rho}^\star$, and $\check{\rho}^\star$ taken with $c=N/n$; $\hat\rho^\circ=\argmin_{\{\rho\in(\max\{0,1-c_N^{-1}\},1]\}}\{\hat D_N(\rho)\}$ and $\check\rho^\circ=\argmin_{\{\rho\in(0,1]\}}\{\check D_N(\rho)\}$.}
  \label{fig:shrinkage_rho}
\end{figure}
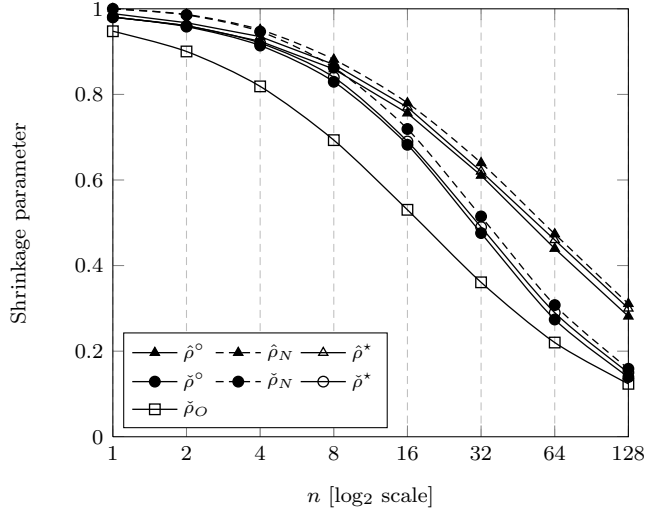

\section{Concluding remarks}
\label{sec:conclusion}
The article shows that, in the large dimensional random matrix regime, the Abramovich--Pascal and Chen estimators for elliptical samples $x_1,\ldots,x_n$ are (up to a variable change) asymptotically equivalent, so that both can be used interchangeably. They are also equivalent to the classical Ledoit--Wolf estimator for the samples $z_1,\ldots,z_n$ or, as can be easily verified, for the samples $\sqrt{N}x_1/\Vert x_1\Vert,\ldots,\sqrt{N}x_n/\Vert x_n\Vert$. This means that for elliptical samples, at least as far as first order convergence is concerned, the Abramovich--Pascal and Chen estimators perform similar to a normalized version of Ledoit--Wolf. 

Recalling that robust estimation theory aims in particular at handling sample sets corrupted by outliers, the performance of the Abramovich--Pascal and Chen estimators given in this paper (not considering outliers) can be seen as a base reference for the ``clean data'' scenario which paves the way for future work in more advanced scenarios. In the presence of outliers, it is expected that the Abramovich--Pascal and Chen estimates exhibit robustness properties that the normalized Ledoit--Wolf scheme does not possess by appropriately weighting good versus outlying data. The study of this scenario is currently under investigation. Also, the extension of this work to second order analysis, e.g., to central limit theorems on linear statistics of the robust estimators, is a direction of future work that will allow to handle more precisely the gain of robust versus non-robust schemes in the not-too-large dimensional regime.

In terms of applications, Proposition~\ref{prop:optimal_shrink} allows for the design of covariance matrix estimators, with minimal Frobenius distance to the population covariance matrix for impulsive i.i.d. samples but in the absence of outliers, and having robustness properties in the presence of outliers. This is fundamental to those scientific fields where the covariance matrix is the object of central interest. More generally though, Theorems~\ref{th:Chitour} and \ref{th:Wiesel} can be used to design optimal covariance matrix estimators under other metrics than the Frobenius norm. This is in particular the case in applications to finance where a possible target consists in the minimization of the risk induced by portfolios built upon such covariance matrix estimates, see e.g., \citep{LED03,RUB12,JIA13}. The possibility to let the number of samples be less than the population size (as opposed to robust estimators of the Maronna-type \citep{MAR76}) is also of interest to applications where optimal shrinkage is not a target but where robustness is fundamental, such as array processing with impulsive noise (e.g., multi-antenna radar) where direction-of-arrival estimates are sought for (see e.g., \citep{MES08c,COU13}). These considerations are also left to future work.

\section{Proofs}
\label{sec:proofs}

This section successively introduces the proofs of Theorem~\ref{th:Chitour}, Theorem~\ref{th:Wiesel}, Corollary~\ref{cor:joint}, Lemma~\ref{lem:equivalent}, Proposition~\ref{prop:shrink}, and Proposition~\ref{prop:optimal_shrink}. The methodology of proof of Theorem~\ref{th:Chitour} closely follows that of \citep{COU13b}. The proof of Theorem~\ref{th:Wiesel} also relies on the same ideas but is more technical due to the imposed normalization of $\check{C}_N(\rho)$ to be of trace $N$. The proofs of the corollary, lemma, and propositions then rely mostly on the important joint convergence over $\rho$ proved in Theorem~\ref{th:Chitour} and Theorem~\ref{th:Wiesel}, and on standard manipulations of random matrix theory and fixed-point equation analysis.

\subsection{Proof of Theorem~\ref{th:Chitour}}
\label{sec:proof_th_Chitour}

The proof of existence and uniqueness of $\hat{C}_N(\rho)$ is given in \citep{PAS13}. 

The existence and uniqueness of $\hat{\gamma}(\rho)$ is quite immediate as the right-hand side integral in the definition of $\hat{\gamma}(\rho)$ is a decreasing function of $\hat{\gamma}$ (since $\rho>0$) with limits $1/(1-\rho)>1$ as $\hat{\gamma}\to 0$ (since $\nu\neq {\bm\delta}_0$ almost everywhere) and zero as $\hat{\gamma}\to\infty$. We now prove the continuity of $\hat\gamma$ on $(0,1]$. Let $\rho_0,\rho\in(0,1]$ and $\hat\gamma_0=\hat\gamma(\rho_0)$, $\hat\gamma=\hat\gamma(\rho)$. Then
\begin{align*}
	\int \frac{t}{\hat\gamma\rho+(1-\rho)t}\nu(dt) - \int\frac{t}{\hat\gamma_0\rho_0+(1-\rho_0)t}\nu(dt) &=0.
\end{align*}
Setting the difference into a common integral and isolating the term $\hat\gamma_0-\hat\gamma$, this becomes, after some calculus,
\begin{align*}
	(\hat\gamma_0-\hat\gamma)\rho_0 &= -\hat\gamma(\rho_0-\rho) + (\rho-\rho_0) \frac{\int \frac{t^2}{(\hat\gamma\rho+(1-\rho)t)(\hat\gamma_0\rho_0+(1-\rho_0)t)}\nu(dt)}{\int \frac{t}{(\hat\gamma\rho+(1-\rho)t)(\hat\gamma_0\rho_0+(1-\rho_0)t)}\nu(dt)}.
\end{align*}
Since the support of $\nu$ is bounded by $\limsup_N\Vert C_N\Vert<\infty$ and in particular $\hat\gamma(\rho)\leq\rho^{-1} \limsup_N\Vert C_N\Vert$ by definition of $\hat\gamma$, the ratio of integrals above is uniformly bounded on $\rho$ in a certain small neighborhood of $\rho_0>0$. Taking the limit $\rho\to\rho_0$ then brings $\hat\gamma_0-\hat\gamma\to 0$, which proves the continuity.

From now on, for readability, we discard all unnecessary indices $\rho$ when no confusion is possible.

Note first that $x_i$ can be equivalently replaced by $z_i$ from the definition of $\hat{C}_N(\rho)$ which is independent of $\tau_1,\ldots,\tau_n$.
Consider $\rho\in\hat{\mathcal R}_\varepsilon$ fixed and assume $\hat{C}_N$ exists for all $N$ on the realization $\{z_1,\ldots,z_n\}_{n=1}^\infty$ (a probability one event). We start by rewriting $\hat{C}_N$ in a more convenient form. Denoting $\hat{C}_{(i)}\triangleq \hat{C}_N-(1-\rho)\frac1n\frac{z_iz_i^*}{\frac1Nz_i^*\hat{C}_N^{-1}z_i}$ and using $(A+tvv^*)^{-1}v=A^{-1}v/(1+tv^*A^{-1}v)$ for positive definite Hermitian $A$, vector $v$, and scalar $t>0$, we have
\begin{align*}
	\frac1Nz_i^*\hat{C}_N^{-1}z_i &= \frac{\frac1Nz_i^*\hat{C}_{(i)}^{-1}z_i}{1+(1-\rho)c\frac{\frac1Nz_i^*\hat{C}_{(i)}^{-1}z_i}{\frac1Nz_i^*\hat{C}_N^{-1}z_i}}
\end{align*}
so that
\begin{align*}
	\frac1Nz_i^*\hat{C}_N^{-1}z_i &= (1-(1-\rho)c_N) \frac1Nz_i^* \hat{C}_{(i)}^{-1}z_i
\end{align*}
and we can rewrite $\hat{C}_N$ as
\begin{align*}
	\hat{C}_N &= \frac{1-\rho}{1-(1-\rho)c_N} \frac1n\sum_{i=1}^n \frac{z_iz_i^*}{\frac1Nz_i^*\hat{C}_{(i)}^{-1}z_i} + \rho I_N.
\end{align*}
The interest of this rewriting is detailed in \citep{COU13b} and mostly lies in the intuition that $\frac1Nz_i^*\hat{C}_{(i)}^{-1}z_i$ should be close to $\frac1N\tr (\hat{C}_N^{-1})$ for all $i$, while $\frac1Nz_i^*\hat{C}_N^{-1}z_i$ is a priori more involved.

To proceed with the proof, for $i\in\{1,\ldots,n\}$, denote $\hat{d}_i(\rho)\triangleq \frac1Nz_i^*\hat{C}_{(i)}^{-1}z_i$ and, up to relabeling, assume $\hat{d}_1(\rho)\leq \ldots \leq \hat{d}_n(\rho)$. Then, using $A\succeq B\Rightarrow B^{-1}\succeq A^{-1}$ for positive Hermitian matrices $A,B$,
\begin{align*}
	\hat{d}_n(\rho) &= \frac1Nz_n^* \left( \frac{1-\rho}{1-(1-\rho)c_N} \frac1n\sum_{i=1}^{n-1} \frac{z_iz_i^*}{\hat{d}_i(\rho)} + \rho I_N \right)^{-1} z_n \\
	&\leq \frac1Nz_n^* \left( \frac{1-\rho}{1-(1-\rho)c_N} \frac1n\sum_{i=1}^{n-1} \frac{z_iz_i^*}{\hat{d}_n(\rho)} + \rho I_N \right)^{-1} z_n.
\end{align*}
Since $z_n\neq 0$, this implies
\begin{align}
	\label{eq:1leq}
	1 &\leq \frac1Nz_n^* \left( \frac{1-\rho}{1-(1-\rho)c_N} \frac1n\sum_{i=1}^{n-1} z_iz_i^* + \hat{d}_n(\rho)\rho I_N \right)^{-1} z_n.
\end{align}
Similarly, with the same derivations, but with opposite inequalities
\begin{align*}
	1 &\geq \frac1Nz_1^* \left( \frac{1-\rho}{1-(1-\rho)c_N} \frac1n\sum_{i=2}^n z_iz_i^* + \hat{d}_1(\rho)\rho I_N \right)^{-1} z_1.
\end{align*}

Our objective is to show that $\sup_{\rho \in \hat{\mathcal R}_\varepsilon}\max_{1\leq i\leq n}|\hat{d}_i(\rho)-{\hat{\gamma}}(\rho)|\asto 0$ where ${\hat{\gamma}}(\rho)$ is given in the statement of the theorem. This is proved via a contradiction argument.

For this, assume that there exists a sequence $\{\rho_n\}_{n=1}^\infty$ over which $\hat{d}_n(\rho_n)>{\hat{\gamma}}(\rho_n)+\ell$ infinitely often, for some $\ell>0$ fixed. Since $\{\rho_n\}_{n=1}^\infty$ is bounded, it has a limit point $\rho_0\in \hat{\mathcal R}_\varepsilon$. Let us restrict ourselves to such a subsequence on which $\rho_n\to\rho_0$ and $\hat{d}_n(\rho_n)>{\hat{\gamma}}(\rho_n)+\ell$. On this sequence, from \eqref{eq:1leq}
\begin{align}
	\label{eq:main_ineq}
	1&\leq \frac1Nz_n^* \left( \frac{1-\rho_n}{1-(1-\rho_n)c_N} \frac1n\sum_{i=1}^{n-1} z_iz_i^* + ({\hat{\gamma}}(\rho_n)+\ell)\rho_n I_N \right)^{-1} z_n \triangleq \hat{e}_n.
\end{align}

Assume first $\rho_0\neq 1$. From standard random matrix results, we have
\begin{align}
	\label{eq:conv_en}
	\hat{e}_n &= \frac{1-(1-\rho_n)c_N}{1-\rho_n} \frac1Nz_n^* \left( \frac1n\sum_{i=1}^{n-1} z_iz_i^* + ({\hat{\gamma}}(\rho_n)+\ell)\rho_n \frac{1-(1-\rho_n)c_N}{1-\rho_n} I_N \right)^{-1} z_n \nonumber \\
	&\asto \frac{1-(1-\rho_0)c}{1-\rho_0} \delta\left( - ({\hat{\gamma}}(\rho_0)+\ell)\rho_0 \frac{1-(1-\rho_0)c}{1-\rho_0} \right) \triangleq e^+
\end{align}
where, for $x>0$, $\delta(x)$ is the unique positive solution to the equation
\begin{align*}
	\delta(x) &= \int \frac{t}{-x + \frac{t}{1+c\delta(x)}} \nu(dt).
\end{align*}
The convergence \eqref{eq:conv_en} follows from several classical ingredients. For this, we first use the fact that, for each $p\geq 2$, $w>0$, and $j\in\{1,\ldots,n\}$, (see e.g., \citep{SIL95,COU13} for similar arguments)
\begin{align*}
	\EE\left[ \left| \frac1Nz_j^* \left( \frac1n\sum_{i\neq j} z_iz_i^* + w I_N \right)^{-1}z_j - \delta (-w) \right|^p  \right] &= \mathcal O\left( N^{-p/2} \right)
\end{align*}
which, taking $p\geq 4$ along with Boole's inequality, Markov inequality, and Borel--Cantelli lemma, ensures that
\begin{align*}
	\max_{1\leq j\leq n} \left| \frac1Nz_j^* \left( \frac1n\sum_{i\neq j} z_iz_i^* + w I_N \right)^{-1}z_j - \delta (-w) \right| &\asto 0.
\end{align*}
Using successively $A^{-1}-B^{-1}=A^{-1}(B-A)B^{-1}$ for invertible $A,B$ matrices and the fact that $\Vert ( \frac1n\sum_{i\neq j} z_iz_i^* + w I_N)^{-1}\Vert < w^{-1}$ and $\limsup_n \max_{1\leq i\leq n}\frac1N\Vert z_i\Vert^2=M_{\nu,1}=1<\infty$ a.s., we then have, for any positive sequence $w_n\to w>0$, 
\begin{align*}
	&\max_{1\leq j\leq n} \left| \frac1Nz_j^* \left( \frac1n\sum_{i\neq j} z_iz_i^* + w_n I_N \right)^{-1}z_j - \frac1Nz_j^* \left( \frac1n\sum_{i\neq j} z_iz_i^* + w I_N \right)^{-1}z_j \right| \nonumber \\
	&= |w_n-w| \max_{1\leq j\leq n} \left| \frac1Nz_j^* \left( \frac1n\sum_{i\neq j} z_iz_i^* + w_n I_N \right)^{-1}\left( \frac1n\sum_{i\neq j} z_iz_i^* + w I_N \right)^{-1}z_j\right| \\
	&\leq |w_n-w|  \frac1{w_nw} \max_{1\leq j\leq n} \frac1N\Vert z_j\Vert^2 \\
	&\asto 0
\end{align*}
from which the convergence \eqref{eq:conv_en} unfolds.

Developing the expression of $e^+$ then leads to $e^+$ being the unique positive solution of the equation
\begin{align*}
	e^+ &= \int \frac{t}{({\hat{\gamma}}(\rho_0)+\ell)\rho_0 + \frac{t}{\frac{1-(1-\rho_0)c}{1-\rho_0} + ce^+} } \nu(dt)
\end{align*}
which we write equivalently
\begin{align}
	\label{eq:e+}
	1 &= \int \frac{t}{({\hat{\gamma}}(\rho_0)+\ell)\rho_0 e^+ + \frac{t e^+}{\frac{1-(1-\rho_0)c}{1-\rho_0} + ce^+} } \nu(dt).
\end{align}
Note that the right-hand side term is a decreasing function $f$ of $e^+$. From the definition of ${\hat{\gamma}}(\rho_0)$, we can in parallel write
\begin{align}
	\label{eq:1}
	1 &= \int \frac{t}{{\hat{\gamma}}(\rho_0) \rho_0 \times 1 + \frac{t \times 1}{\frac{1-(1-\rho_0)c}{1-\rho_0} + c\times 1} } \nu(dt)
\end{align}
where we purposely made the terms $1$ explicit. Now, since both integrals above equal $1$, since $\ell>0$, and since $f$ is decreasing, we must have $e^+<1$. But this is in contradiction with $\hat{e}_n\geq 1$ and the convergence \eqref{eq:conv_en}. 

If instead, $\rho_0=1$, then from the definition of $\hat{e}_n$ in \eqref{eq:main_ineq}, and since $\frac1N\Vert z_n\Vert^2\asto M_{\nu,1}=1$ (from $\lim_n \max_{1\leq i\leq n} | \frac1N\Vert z_i\Vert^2 - M_{\nu,1} | \asto 0$), $\limsup_n \Vert \frac1n\sum_{i=1}^ n z_iz_i^*\Vert<\infty$ a.s. (from Assumption~\ref{ass:x}--\ref{item:C} and \citep{SIL98}), and $\hat{\gamma}(1)=M_{\nu,1}=1$, we have
\begin{align*}
	\hat{e}_n &\asto \frac{M_{\nu,1}}{M_{\nu,1}+\ell} = \frac{1}{1+\ell} < 1
\end{align*}
again contradicting $\hat{e}_n\geq 1$.

Hence, for all large $n$, there is no sequence of $\rho_n$ for which $\hat{d}_n(\rho_n) > {\hat{\gamma}}(\rho_n)+\ell$ infinitely often and therefore $\hat{d}_n(\rho)\leq {\hat{\gamma}}(\rho)+\ell$ for all large $n$ a.s., uniformly on $\rho\in\hat{\mathcal R}_\varepsilon$.

The same reasoning holds for $\hat{d}_1(\rho)$ which can be proved greater than ${\hat{\gamma}}(\rho)-\ell$ for all large $n$ uniformly on $\rho\in\hat{\mathcal R}_\varepsilon$. Consequently, since $\ell>0$ is arbitrary, from the ordering of the $\hat{d}_i(\rho)$, we have proved that $\sup_{\rho\in\hat{\mathcal R}_\varepsilon}\max_{1\leq i\leq n}|\hat{d}_i(\rho)-{\hat{\gamma}}(\rho)|\asto 0$.

From there, we then find that
\begin{align*}
	\sup_{\rho \in \hat{\mathcal R}_\varepsilon}\left\Vert \hat{S}_N(\rho) - \hat{C}_N(\rho) \right\Vert &\leq \left\Vert \frac1n\sum_{i=1}^n z_iz_i^* \right\Vert \sup_{\rho \in \hat{\mathcal R}_\varepsilon} \max_{1\leq i\leq n}  \frac{1-\rho}{1-(1-\rho)c_N} \left| \frac{\hat{d}_i(\rho)-{\hat{\gamma}(\rho)}}{{\hat{\gamma}(\rho)} \hat{d}_i(\rho)} \right| \\
	&\asto 0
\end{align*}
where we used the fact that $\limsup_n\left\Vert \frac1n\sum_{i=1}^n z_iz_i^* \right\Vert < \infty$ a.s. from Assumption~\ref{ass:x}--\ref{item:C} and \citep{SIL98}, and the fact that $0<\varepsilon<c^{-1}$.

\subsection{Proof of Theorem~\ref{th:Wiesel}}
\label{sec:proof_th_Wiesel}

The proof of existence and uniqueness is given in \citep{CHE11}. The proof of Theorem~\ref{th:Wiesel} unfolds similarly as the proof of Theorem~\ref{th:Chitour} but it slightly more involved due to the difficulty brought by the normalization of $\check{C}_N(\rho)$ by its trace. For this reason, we first introduce some preliminary results needed in the main core of the proof. Note also that, similar to the proof of Theorem~\ref{th:Chitour}, we may immediately consider $z_i$ in place of $x_i$ in the expression of $\check{C}_N(\rho)$ from the independence of $\check{C}_N(\rho)$ with respect to $\tau_1,\ldots,\tau_n$.

From now on, for the sake of readability, we discard the unnecessary indices $\rho$.

\subsubsection{Some preliminaries}

We start by some considerations on ${\check{\gamma}}(\rho)$ and $F_N(x)$ defined as the unique positive solution to the equation in $F_N$
\begin{align}
	\label{eq:FN}
	F_N=(1-\rho)\frac1x \frac1{F_N}+\rho-c_N(1-\rho).
\end{align}
Note first that, for $x>0$, \eqref{eq:FN} can be written as a second order polynomial whose solutions have opposite signs, the positive one being explicitly given by
	\begin{align*}
		F_N(x) &= \frac12\left(\rho - c_N(1-\rho) \right) + \sqrt{\frac14\left(\rho - c_N(1-\rho) \right)^2+(1-\rho)\frac1x}.
	\end{align*}
	The function $F_N(x)$ is decreasing with $\lim_{x\to 0}F_N(x)=\infty$ and $\lim_{x\to\infty}F_N(x)=\max \{\rho-c_N(1-\rho),0\}$. As $N\to\infty$, $c_N\to c$, and $F_N(x)\to F(x)=F(x;\rho)$ defined in the statement of the theorem which therefore satisfies $F(x)=(1-\rho)\frac1x \frac1{F(x)}+\rho-c(1-\rho)$ and is decreasing with $\lim_{x\to 0}F(x)=\infty$ and $\lim_{x\to\infty}F(x)=\max \{\rho-c(1-\rho),0\}$. This implies in particular that the function
\begin{align}
	\label{eq:fun_G}
	G: x \mapsto \int \frac{t}{x\rho + \frac{1-\rho}{(1-\rho)c+F(x)}t}\nu(dt)
\end{align}
is decreasing with $\lim_{x\to 0}G(x)=\infty$ and $\lim_{x\to\infty}G(x)=0$. Hence the existence and uniqueness of ${\check{\gamma}}(\rho)$ as defined in the theorem.

Now consider the function $H_N:x\mapsto xF_N(x)$ for $x>0$ and $\rho<1$. Then, for $x>0$,
\begin{align*}
	H_N'(x) &= \frac12\frac{A(x)+B(x)}{\sqrt{\left( \frac{\rho-(1-\rho)c_N}2 \right)^2x^2+(1-\rho)x}}
\end{align*}
where
\begin{align*}
	A(x)&= 2\left( \frac{\rho-(1-\rho)c_N}2 \right)\sqrt{\left( \frac{\rho-(1-\rho)c_N}2 \right)^2x^2+(1-\rho)x} \\
	B(x)&= 1-\rho + 2 \left( \frac{\rho-(1-\rho)c_N}2 \right)^2x.
\end{align*}
Although $A(x)$ may be negative, it is easily verified that $B(x)^2=A(x)^2+(1-\rho)^2$ for all $x\geq 0$. Therefore, if $\rho<1$, for each $w_0>0$, there exists $\varepsilon>0$ such that
\begin{align}
	\label{eq:uniform_HN'}
	\liminf_N \sup_{w_0-\varepsilon<x<w_0+\varepsilon} H_N'(x) > 0
\end{align}
a relation which will be useful in the core of the proof of Theorem~\ref{th:Wiesel}.

To prove continuity of ${\check{\gamma}}$, the same arguments as in the proof of Theorem~\ref{th:Chitour} hold. That is, take $\rho_0,\rho\in(0,1]$ and denote $\check\gamma_0=\check\gamma(\rho_0)$ and $\check\gamma=\check\gamma(\rho)$. Then, by definition of $\check\gamma(\rho)$, using $F(x)=(1-\rho)\frac1x \frac1{F(x)}+\rho-c(1-\rho)$,
\begin{align*}
	\int \frac{t}{\check\gamma_0\rho_0+\frac{(1-\rho_0)\check\gamma_0F(\check\gamma_0)}{1-\rho_0+\rho_0\check\gamma_0 F(\check{\gamma}_0)}t}\nu(dt) - \int \frac{t}{\check\gamma\rho+\frac{(1-\rho)\check\gamma F(\check\gamma)}{1-\rho+\rho \check\gamma F(\check{\gamma})}t}\nu(dt) &= 0.
\end{align*}
Setting these to a common denominator gives, after some calculus,
\begin{align}
	\label{eq:continuity_check}
&	\left[ (\check\gamma_0-\check\gamma)\rho_0+\check\gamma(\rho_0-\rho)\right]\int \frac{t}{D(t)}\nu(dt) \nonumber\\
&= \frac{(1-\rho)(1-\rho_0)(\check\gamma F(\check\gamma)-\check\gamma_0 F(\check\gamma_0))+(\rho_0-\rho)\check\gamma\check\gamma_0F(\check\gamma)F(\check\gamma_0)}{(1-\rho+\rho\check\gamma F(\check\gamma))(1-\rho_0+\rho_0\check\gamma_0 F(\check\gamma_0))} \int \frac{t^2}{D(t)}\nu(dt)
\end{align}
where 
\begin{align*}
D(t)=\left(\check\gamma_0\rho_0+\frac{(1-\rho_0)\check\gamma_0F(\check\gamma_0)}{1-\rho_0+\rho_0\check\gamma_0 F(\check{\gamma}_0)}t\right)\left(\check\gamma\rho+\frac{(1-\rho)\check\gamma F(\check\gamma)}{1-\rho+\rho \check\gamma F(\check{\gamma})}t\right)>0.
\end{align*}
Note now that $\check\gamma(\rho)\leq \rho^{-1}\limsup_N\Vert C_N\Vert$ and, on a small neighborhood of $\rho_0\in(0,1]$, $\check\gamma=\check\gamma(\rho)$ is uniformly away from zero. Indeed, if this were not the case, on some subsequence $\rho_k\to\rho_0$ such that $\check\gamma(\rho_k)\to 0$, the definition of $\check\gamma$ would imply 
\begin{align*}
	1=\int \frac{t}{\check\gamma(\rho_k) \rho_k+\frac{1-\rho}{(1-\rho_k)c+F(\check\gamma(\rho_k))}}\nu(dt)\to 0
\end{align*}
which is a contradiction. This implies as a consequence that $F(\check\gamma)$ is bounded on a neighborhood of $\rho_0$. All this implies that all terms proportional to $\rho_0-\rho$ in \eqref{eq:continuity_check} tend to zero as $\rho\to\rho_0$, so that, in the limit $\rho\to\rho_0$,
\begin{align*}
	(\check\gamma_0-\check\gamma)\rho_0 \int \frac{t\nu(dt)}{D(t)} + \frac{(1-\rho)(1-\rho_0)(\check\gamma_0 F(\check\gamma_0)-\check\gamma F(\check\gamma))}{(1-\rho+\rho\check\gamma F(\check\gamma))(1-\rho_0+\rho_0\check\gamma_0 F(\check\gamma_0))} \int \frac{t^2\nu(dt)}{D(t)} &\to 0.
\end{align*}
But, since $x\mapsto xF(x)$ is increasing, $\check\gamma_0 F(\check\gamma_0)-\check\gamma F(\check\gamma)$ is of the same sign as $\check\gamma_0-\check\gamma$. As $D(t)$ is uniformly bounded for $\rho$ in a small neighborhood of $\rho_0$, this induces $\check\gamma_0-\check\gamma\to 0$, which concludes the proof of continuity.

\subsubsection{Main proof}

Let us now work on the matrix $\check{B}_N$. From the definition of $\check{C}_N$,
\begin{align*}
	\check{B}_N &= \frac{1-\rho}{\frac1N\tr \check{B}_N} \frac1n\sum_{i=1}^n \frac{z_iz_i^*}{\frac1Nz_i^*\check{B}_N^{-1}z_i} + \rho I_N.
\end{align*}

Denoting $\check{B}_{(i)}=\check{B}_N-\frac{1-\rho}{\frac1N\tr \check{B}_N}\frac1n\frac{z_iz_i^*}{\frac1Nz_i^*\check{B}_N^{-1}z_i}$ and using again $(A+txx^*)^{-1}x=A^{-1}x/(1+tx^*A^{-1}x)$, we have this time
\begin{align*}
	\frac1Nz_i^*\check{B}_N^{-1}z_i &= \frac{\frac1Nz_i^*\check{B}_{(i)}^{-1}z_i}{1+ (1-\rho)c_N\frac{\frac1Nz_i^*\check{B}_{(i)}^{-1}z_i}{\frac1Nz_i^*\check{B}_N^{-1}z_i} \frac1{\frac1N\tr \check{B}_N} }
\end{align*}
so that
\begin{align}
	\label{eq:BN_Bi}
	\frac1Nz_i^*\check{B}_N^{-1} z_i &= \frac1Nz_i^*\check{B}_{(i)}^{-1}z_i \left( 1 - c_N(1-\rho) \frac1{\frac1N\tr \check{B}_N} \right).
\end{align}
From the positivity of both quadratic forms above, this implies in particular that $\frac1N\tr \check{B}_N-c(1-\rho)>0$.

Replacing the quadratic forms $\frac1Nz_i^*\check{B}_N^{-1} z_i$ in the expression of $\check{B}_N$, we can now rewrite $\check{B}_N$ as
\begin{align}
	\label{eq:B}
	\check{B}_N &= \frac{1-\rho}{\frac1N\tr \check{B}_N - c_N(1-\rho)} \frac1n\sum_{i=1}^n \frac{z_iz_i^*}{\frac1Nz_i^*\check{B}_{(i)}^{-1}z_i} + \rho I_N.
\end{align}

Denote now $\check{d}_i\triangleq \frac1Nz_i^*\check{B}_{(i)}^{-1}z_i$ and assume, up to relabeling, that $\check{d}_1\leq \ldots \leq \check{d}_n$ for all $n$. Then, with the definition of $\check{B}_{(i)}$, we have
\begin{align*}
	\check{d}_n &= \frac1Nz_n^* \left( \frac{1-\rho}{\frac1N\tr \check{B}_N - c_N(1-\rho)} \frac1n\sum_{i=1}^{n-1} \frac{z_iz_i^*}{\check{d}_i} + \rho I_N \right)^{-1}z_n \\
	&\leq \frac1Nz_n^* \left( \frac{1-\rho}{\frac1N\tr \check{B}_N - c_N(1-\rho)} \frac1n\sum_{i=1}^{n-1} \frac{z_iz_i^*}{\check{d}_n} + \rho I_N \right)^{-1}z_n \\
	&= \frac{\frac1N\tr \check{B}_N - c_N(1-\rho)}{1-\rho} \frac1Nz_n^* \left( \frac1n\sum_{i=1}^{n-1} \frac{z_iz_i^*}{\check{d}_n} + \rho \frac{\frac1N\tr \check{B}_N - c_N(1-\rho)}{1-\rho} I_N \right)^{-1}z_n
\end{align*}
where the inequality follows from the initial quadratic form being increasing when seen as a function of $\check{d}_i$ for each $i$. 
This can be equivalently written
\begin{align}
	\label{eq:1leqb}
	1 &\leq \frac{\frac1N\tr \check{B}_N - c_N(1-\rho)}{1-\rho} \frac1Nz_n^* \left( \frac1n\sum_{i=1}^{n-1} z_iz_i^* + \check{d}_n\rho \frac{\frac1N\tr \check{B}_N - c_N(1-\rho)}{1-\rho} I_N \right)^{-1}z_n.
\end{align}

At this point, it is convenient to express \eqref{eq:1leqb} as a function of $F_N$ defined in \eqref{eq:FN}. From \eqref{eq:B}, note indeed that
\begin{align*}
	\frac1N\tr \check{B}_N &= \frac{1-\rho}{\frac1N\tr \check{B}_N - c_N(1-\rho)} \frac1n\sum_{i=1}^n\frac{\frac1N\Vert z_i\Vert^2}{\check{d}_i} + \rho
\end{align*}
so that, since $\frac1N\tr \check{B}_N - c_N(1-\rho)>0$,
\begin{align}
	\label{eq:B_FN}
	\frac1N\tr \check{B}_N - c_N(1-\rho) &= F_N\left( \left[ \frac1n\sum_{i=1}^{n}\frac{\frac1N\Vert z_i\Vert^2}{\check{d}_i} \right]^{-1} \right).
\end{align}
Since $F_N$ is decreasing, the term on the right-hand side is decreasing in $\check{d}_i$ for each $i$. Hence
\begin{align*}
	F_N\left( \left[ \frac1n\sum_{i=1}^{n}\frac{\frac1N\Vert z_i\Vert^2}{\check{d}_i} \right]^{-1} \right) &\geq F_N\left( \check{d}_n \left[  \frac1n\sum_{i=1}^{n}\frac1N\Vert z_i\Vert^2 \right]^{-1} \right).
\end{align*}

This implies, returning to \eqref{eq:1leqb}
\begin{align}
	\label{eq:1leqc}
	1 &\leq \frac1{1-\rho} F_N\left( \left[ \frac1n\sum_{i=1}^{n}\frac{\frac1N\Vert z_i\Vert^2}{\check{d}_i} \right]^{-1} \right) \nonumber \\
	&\times \frac1Nz_n^* \left( \frac1n\sum_{i=1}^{n-1} z_iz_i^* + \check{d}_n \frac{\rho}{1-\rho} F_N\left( \check{d}_n \left[  \frac1n\sum_{i=1}^{n}\frac1N\Vert z_i\Vert^2 \right]^{-1} \right) I_N \right)^{-1}z_n.
\end{align}

With this, similar to the proof of Theorem~\ref{th:Chitour}, we will now show via a contradiction argument that $\sup_{\rho\in\check{\mathcal R}_\varepsilon}\max_{1\leq i\leq n}|\check{d}_i(\rho)-\check{\gamma}(\rho)|\asto 0$. 
Let us then assume that, on a sequence $\{\rho_n\}_{n=1}^\infty$, $\check{d}_n=\check{d}_n(\rho_n) > {\check{\gamma}}(\rho_n) + \ell=\check{\gamma}+\ell$ infinitely often, for some $\ell>0$, and let us consider a subsequence on which $\rho_n\to \rho_0\in\check{\mathcal R}_\varepsilon$ and $\check{d}_n(\rho_n)>{\check{\gamma}}(\rho_n)+\ell$. Then, from the fact that $H_N(x)=xF_N(x)$ is increasing for $x>0$, we have
\begin{align}
	\label{eq:1_leq_gamma_ell}
	1 &\leq \frac1{1-\rho} F_N\left( \left[ \frac1n\sum_{i=1}^{n}\frac{\frac1N\Vert z_i\Vert^2}{\check{d}_i} \right]^{-1} \right) \nonumber \\
	&\times \frac1Nz_n^* \left( \frac1n\sum_{i=1}^{n-1} z_iz_i^* + \frac{({\check{\gamma}}+\ell)\rho}{1-\rho} F_N\left( ({\check{\gamma}}+\ell) \left[  \frac1n\sum_{i=1}^{n}\frac1N\Vert z_i\Vert^2 \right]^{-1} \right) I_N \right)^{-1}z_n.
\end{align}

Assume first that $\rho_0<1$. We will deal with each factor involving $F_N$ on the right-hand side of \eqref{eq:1_leq_gamma_ell}. We start with the right-most factor.
Using $\max_{1\leq i\leq n} \{\frac1N\Vert z_i\Vert^2\}\asto 1$ since $\frac1N\tr C_N=1$ for each $N$, $\check{\gamma}(\rho_n)\to \check{\gamma}(\rho_0)$ (by continuity of $\check{\gamma}$) and also the fact that $\lim_N \inf_{\{\check{\gamma}(\rho_0)-\eta<x<\check{\gamma}(\rho_0)+\eta\}} H_N'(x)>0$ for some $\eta>0$ small (from \eqref{eq:uniform_HN'}), from classical random matrix theory results, e.g., \citep{SIL95}, we obtain, with probability one
\begin{align}
	&\lim_n \frac1Nz_n^* \left( \frac1n\sum_{i=1}^{n-1} z_iz_i^* + \frac{({\check{\gamma}}+\ell)\rho_n}{1-\rho_n} F_N\left( ({\check{\gamma}}+\ell) \left[  \frac1n\sum_{i=1}^{n}\frac1N\Vert z_i\Vert^2 \right]^{-1} \right) I_N \right)^{-1}z_n \nonumber \\
	&<\lim_n \frac1Nz_n^* \left( \frac1n\sum_{i=1}^{n-1} z_iz_i^* + \frac{{\check{\gamma}}\rho_n}{1-\rho_n} F_N\left( {\check{\gamma}} \left[  \frac1n\sum_{i=1}^{n}\frac1N\Vert z_i\Vert^2 \right]^{-1} \right) I_N \right)^{-1}z_n \label{eq:strict_ineq} \\
	&= \delta \nonumber
\end{align}
where $\delta$ is the unique positive solution to
\begin{align*}
	\delta &= \int \frac{t}{\frac{\rho_0{\check{\gamma}(\rho_0)} F({\check{\gamma}}(\rho_0))}{1-\rho_0}+\frac{t}{1+c \delta}}\nu(dt).
\end{align*}
Note here the fundamental importance of having $H_N'$ uniformly positive in a neighborhood of $\check{\gamma}(\rho_0)$ to ensure the inequality sign in \eqref{eq:strict_ineq} remains strict when passing to the limit over $n$.
We will now show that $e\triangleq \frac{F({\check{\gamma}}(\rho_0))}{1-\rho_0}\delta=1$. Indeed, from the above equation,
\begin{align*}
	e &= \int \frac{t}{ \rho_0{\check{\gamma}(\rho_0)}+\frac{(1-\rho_0)t}{F({\check{\gamma}(\rho_0)})+(1-\rho_0)c e}}\nu(dt)
\end{align*}
or equivalently
\begin{align}
	\label{eq:1_of_e}
	1 &= \int \frac{t}{e\rho_0{\check{\gamma}(\rho_0)}+\frac{(1-\rho_0)te}{F({\check{\gamma}(\rho_0)})+(1-\rho_0)c e}}\nu(dt).
\end{align}
The right-hand side of \eqref{eq:1_of_e} is a decreasing function of $e$ with limits $\infty$ as $e\to 0$ and $0$ as $e\to\infty$. As an equation of $e$, \eqref{eq:1_of_e} therefore has a unique positive solution which happens to be $1$ by definition of ${\check{\gamma}}(\rho_0)$ in the theorem statement. Therefore, $e=1$.

Now consider the leading factor involving $F_N$ in \eqref{eq:1_leq_gamma_ell}. We will show that this factor is uniformly bounded. For this, proceeding similarly as above with $\check{d}_1$ instead of $\check{d}_n$, note that \eqref{eq:1leqc}, with $\rho=\rho_n$, becomes (this is obtained by reverting all inequality signs in the preceding derivations)
\begin{align}
	\label{eq:1leqd}
	1 &\geq \frac1{1-\rho_n} F_N\left( \left[ \frac1n\sum_{i=1}^{n}\frac{\frac1N\Vert z_i\Vert^2}{\check{d}_i} \right]^{-1} \right) \nonumber \\
	&\times\frac1Nz_1^* \left( \frac1n\sum_{i=1}^{n-1} z_iz_i^* + \check{d}_1 \frac{\rho_n}{1-\rho_n} F_N\left( \check{d}_1 \left[  \frac1n\sum_{i=1}^{n}\frac1N\Vert z_i\Vert^2 \right]^{-1} \right) I_N \right)^{-1}z_1.
\end{align}
Assume $\frac1n\sum_{i=1}^n\frac{\frac1N\Vert z_i\Vert^2}{\check{d}_i}\to \infty$ on some subsequence (of probability one) over which $\max_i \frac1N\Vert z_i\Vert^2\to 1$. In particular $\check{d}_1\to 0$. Then, from the limiting values taken by $F_N$ and $H_N$, the quadratic form in \eqref{eq:1leqd} has positive limit (even infinite if $c>1$) while the first term on the right-hand side tends to infinity. This contradicts \eqref{eq:1leqd} altogether and therefore $\limsup_n\frac1n\sum_{i=1}^n\frac{\frac1N\Vert z_i\Vert^2}{\check{d}_i}<\infty$.

Since in addition $\check{d}_i\leq \rho_n^{-1} \frac1N\Vert z_i\Vert^2$ (using $\Vert (A+\rho_n I_N)^{-1}\Vert\leq \rho_n^{-1}$ for nonnegative Hermitian $A$) is uniformly bounded a.s. for all large $n$, it follows that $\frac1n\sum_{i=1}^n\frac{\frac1N\Vert z_i\Vert^2}{\check{d}_i}$ is uniformly bounded and bounded away from zero. This implies that $F_N\left( \left[ \frac1n\sum_{i=1}^{n}\frac{\frac1N\Vert z_i\Vert^2}{\check{d}_i} \right]^{-1} \right)$ is uniformly bounded, as desired.

Getting back to \eqref{eq:1_leq_gamma_ell} with $\rho=\rho_n$, we can therefore extract a further subsequence on which the latter converges to $F^\infty$ and $\check{d}_1$ converges to $\check{d}_1^\infty$ ($\check{d}_1^\infty$ can be zero) and we then have along this subsequence
\begin{align}
	\label{eq:cruz_equation}
	1 &< \frac{F^\infty}{1-\rho_0} \delta = \frac{F^\infty}{F({\check{\gamma}}(\rho_0))}
\end{align}
with the equality arising from $F(\check{\gamma}(\rho_0))\delta=1-\rho_0$.

Since $F_N$ is increasing,
\begin{align*}
	F_N\left( \left[ \frac1n\sum_{i=1}^{n}\frac{\frac1N\Vert z_i\Vert^2}{\check{d}_i} \right]^{-1} \right) &\leq F_N\left( \check{d}_i \left[ \frac1n\sum_{i=1}^n \frac1N\Vert z_i\Vert^2 \right]^{-1}\right)
\end{align*}
so that, taking the limit over $n$, $F^\infty\leq F(\check{d}_1^\infty)$ (set equal to $\infty$ if $\check{d}_1^\infty=0$). This further implies
\begin{align*}
	F({\check{\gamma}}(\rho_0)) &< F(\check{d}_1^\infty)
\end{align*}
so that, if $\check{d}_1^\infty>0$, inverting the above inequality, gives $\check{d}_1^\infty<{\check{\gamma}}(\rho_0)$. Obviously, if $\check{d}_1^\infty=0$, this is still true. Therefore $\check{d}_1(\rho_n)<{\check{\gamma}}(\rho_0)-\ell'$ infinitely often for some $\ell'>0$ along the considered subsequence.

Conserving the same subsequence and reproducing the same steps for the sequence $\check{d}_1(\rho_n)$ instead of $\check{d}_n(\rho_n)$ (from \eqref{eq:1leqd}, use $\check{d}_1(\rho_n)<{\check{\gamma}}(\rho_n)-\ell'$ infinitely often and the growth of $H_N$ similar to before), we obtain this time
\begin{align*}
	1 &> \frac{F^\infty}{F({\check{\gamma}}(\rho_0))}
\end{align*}
which contradicts \eqref{eq:cruz_equation}. 

Assume now $\rho_0=1$. 
Starting from \eqref{eq:1leqb} with $\rho=\rho_n$ and the expression of $F_N$, we have
\begin{align*}
	1 &\leq \limsup_N F_N\left( \left[ \frac1n\sum_{i=1}^n \frac{\frac1N\Vert z_i\Vert^2}{\check{d}_i} \right]^{-1} \right) \nonumber \\ 
	&\times \frac1Nz_n^* \left( (1-\rho_n) \frac1n\sum_{i=1}^{n-1} z_iz_i^* + \check{d}_n\rho_n F_N\left( \left[ \frac1n\sum_{i=1}^n \frac{\frac1N\Vert z_i\Vert^2}{\check{d}_i} \right]^{-1} \right)  I_N \right)^{-1}z_n \\
	&\leq \limsup_N F_N\left( \left[ \frac1n\sum_{i=1}^n \frac{\frac1N\Vert z_i\Vert^2}{\check{d}_i} \right]^{-1} \right) \nonumber \\ 
	&\times \frac1Nz_n^* \left( (1-\rho_n) \frac1n\sum_{i=1}^{n-1} z_iz_i^* + (\check{\gamma}+\ell)\rho_n F_N\left( \left[ \frac1n\sum_{i=1}^n \frac{\frac1N\Vert z_i\Vert^2}{\check{d}_i} \right]^{-1} \right)  I_N \right)^{-1}z_n \\
	&= \frac1{\check{\gamma}(\rho_0)+\ell}
\end{align*}
since $\rho_n\to \rho_0=1$, since $\frac1n\sum_{i=1}^n \frac{\frac1N\Vert z_i\Vert^2}{\check{d}_i}$ is uniformly away from zero (as shown previously), and since $\limsup_n\Vert \frac1n\sum_{i=1}^nz_iz_i^*\Vert<\infty$ \citep{SIL98}. But then, the fact that $\check{\gamma}(\rho_0)=1$ by definition along with the above relation leads to $1\leq 1/(1+\ell)$, again a contradiction.

Therefore, gathering the results, our very initial hypothesis that there exists a subsequence of $n$ and $\rho_n$ over which $\check{d}_n(\rho_n)>\gamma(\rho_n)+\ell$ infinitely often is invalid and we conclude that, instead, $\sup_{\rho \in \check{\mathcal R}_\varepsilon}\check{d}_n(\rho)-{\check{\gamma}}(\rho)\leq \ell$ for all large $n$ a.s.

The same procedure works similarly when starting over with $\check{d}_1$ and assuming with the same contradiction argument that $\check{d}_1(\rho'_n)<{\check{\gamma}}(\rho'_n)-\ell$ infinitely often on some sequence $\rho'_n$. Taking a subsequence over which $\rho'_n\to \rho'_0$, this will imply this time that $\check{d}_n(\rho'_0)>{\check{\gamma}}(\rho'_0)+\ell'$ for some $\ell'>0$ for all large $n$ a.s. which we now know is invalid. 

Gathering the results, we finally obtain
\begin{align}
	\label{eq:conv_checkd}
	\sup_{\rho\in\check{\mathcal R}_\varepsilon}\max_{1\leq i\leq n} |\check{d}_i(\rho)-{\check{\gamma}}(\rho)| \asto 0
\end{align}
as desired. 
This implies from \eqref{eq:B_FN} that
\begin{align*}
	\sup_{\rho\in\check{\mathcal R}_\varepsilon} \left|\frac1N\tr \check{B}_N - c(1-\rho) - F({\check{\gamma}}(\rho)) \right| \asto 0
\end{align*}
with $\inf_{\rho\in\check{\mathcal R}_\varepsilon}F({\check{\gamma}}(\rho))>0$ so that, from \eqref{eq:B}, Assumption~\ref{ass:x}--\ref{item:C}, and \citep{SIL98},
\begin{align*}
	\sup_{\rho\in\check{\mathcal R}_\varepsilon} \left\Vert \check{B}_N - \left[\frac{1-\rho}{F({\check{\gamma}(\rho)}){\check{\gamma}(\rho)}}\frac1n\sum_{i=1}^n z_iz_i^* + \rho I_N\right] \right\Vert \asto 0.
\end{align*}
Dividing the expression inside the norm by $\frac1N\tr \check{B}_N$ and taking the limit finally gives
\begin{align*}
	\sup_{\rho\in\check{\mathcal R}_\varepsilon} \left\Vert \check{C}_N - \left[\frac{1-\rho}{\rho F({\check{\gamma}}){\check{\gamma}} + (1-\rho)}\frac1n\sum_{i=1}^n z_iz_i^* + \frac{\rho {\check{\gamma}} F({\check{\gamma}})}{\rho{\check{\gamma}} F({\check{\gamma}}) + (1-\rho) } I_N\right] \right\Vert \asto 0
\end{align*}
with $\check{\gamma}=\check{\gamma}(\rho)$, which is the expected result.

\subsection{Proof of Corollary~\ref{cor:joint}}
\label{sec:proof_joint}
	We only give the proof for $\hat{C}_N(\rho)$. Similar arguments hold for $\check{C}_N(\rho)$. The joint eigenvalue convergence is an application of \citep[Theorem~4.3.7]{HOR85} on the spectral norm convergence of Theorems~\ref{th:Chitour} and \ref{th:Wiesel}. The norm boundedness results from $\sup_{\rho\in \hat{\mathcal R}_\varepsilon} | \Vert \hat{C}_N(\rho) \Vert - \Vert \hat{S}_N(\rho)\Vert |\asto 0$ and from $\limsup_N\sup_{\rho\in \hat{\mathcal R}_\varepsilon} \Vert \hat{S}_N(\rho)\Vert <\infty$ by an application of \citep{SIL98}.
	The joint convergence of moments over $\hat{\mathcal R}_\varepsilon$ follows first from the convergence $\hat{m}_N(z;\rho) - m_{\hat{\mu}_\rho}(z)\asto 0$ for each $z$ with $\Im[z]>0$ and for each $\rho\in\hat{\mathcal R}_\varepsilon$  where $m_N(z;\rho)=\frac1N\tr( (\hat{S}_N(\rho) - z I_N)^{-1})$ (as a consequence of Corollary~\ref{cor:limit}). Since this holds for each such $z$, the almost sure convergence is also valid uniformly on a countable set of $z$ with $\Im[z]>0$ having a limit point away from the union $\mathcal U$ over $\rho\in\hat{\mathcal R}_\varepsilon$ of the limiting spectra of $\hat{S}_N(\rho)$, $\mathcal U$ being a bounded set since $\limsup_N \sup_{\rho\in\hat{\mathcal R}_\varepsilon} \Vert \hat{S}_N(\rho)\Vert<\infty$. But then, since
	\begin{align*}
		\frac{(1-\rho)m_N(z;\rho)}{\hat{\gamma}(\rho)(1-(1-\rho)c)} &= \frac1N\tr \left[\left( \frac1n\sum_{i=1}^n z_iz_i^* + \frac{\rho-z}{1-\rho}\hat{\gamma}(\rho)(1-(1-\rho)c) I_N\right)^{-1}\right]
	\end{align*}
	is analytic in $\hat{z}(\rho)=\frac{\rho-z}{1-\rho}\hat{\gamma}(\rho)(1-(1-\rho)c)$ and bounded on all bounded regions away from $\mathcal U$, by Vitali's convergence theorem \citep{TIT39}, the convergence $\hat{m}_N(z;\rho) - m_{\hat{\mu}_\rho}(z)\asto 0$ is uniform on such bounded sets of $(z,\rho)$. Using the Cauchy integrals $\oint z^km_N(z;\rho)dz=\frac1N\tr (\hat{S}_N(\rho)^\ell)$ and $\oint z^km_{\hat{\mu}_\rho}(z)dz=M_{\hat{\mu}_\rho,k}$ for each $k\in\NN$ on a contour that circles around (but sufficiently away from) $\mathcal U$ implies $\sup_{\rho\in\hat{\mathcal R}_\varepsilon} |\frac1N\tr (\hat{S}_N(\rho)^\ell) - M_{\hat{\mu}_\rho,\ell}|\asto 0$, from which the result unfolds.

\subsection{Proof of Lemma~\ref{lem:equivalent}}
\label{sec:proof_lemma}
We start with $\hat{S}_N$. Remark first that, for $\rho\in (\max\{0,1-c^{-1}\},1]$,
\begin{align*}
	\frac{\hat{S}_N(\rho)}{M_{\hat\mu_{\rho},1}} &= \left(1-\frac{\rho}{\frac1{\hat\gamma(\rho)}\frac{1-\rho}{1-(1-\rho)c}+\rho}\right) \frac1n\sum_{i=1}^n z_iz_i^* + \frac{\rho}{\frac1{\hat\gamma(\rho)}\frac{1-\rho}{1-(1-\rho)c}+\rho}I_N.
\end{align*}
Denoting
\begin{align*}
\hat f:(\max\{0,1-c^{-1}\},1] &\to (0,1] \\
	\rho&\mapsto \frac{\rho}{\frac1{\hat\gamma(\rho)} \frac{1-\rho}{1-(1-\rho)c} + \rho} = \frac1{\frac1{\rho\hat\gamma(\rho)} \frac{1-\rho}{1-(1-\rho)c} + 1}
\end{align*}
we have $\frac{\hat{S}_N(\rho)}{M_{\hat\mu_{\rho},1}}=(1-\hat{f}(\rho))\frac1n\sum_{i=1}^n z_iz_i^* + \hat{f}(\rho)I_N$ and it therefore suffices to show that $\hat f$ is continuously increasing and onto. The continuity of $\hat f$ unfolds immediately from the continuity of $\hat\gamma$. By the definition of $\hat\gamma$, the function $\rho\mapsto \rho\hat\gamma(\rho)$ is increasing and nonnegative (since $\nu$ is distinct from ${\bm\delta}_0$ almost everywhere) while $\rho\mapsto \frac{1-\rho}{1-(1-\rho)c}$ is decreasing and nonnegative. Therefore, $\hat f$ is increasing and nonnegative. It remains to show that $\hat f$ is onto. Clearly $\hat f(1)=1$ since $\hat\gamma(1)=M_{\nu,1}=1$. 
To handle the lower limit, let us rewrite
\begin{align*}
	\hat f(\rho) &=  \frac{\rho\hat{\gamma}(\rho)(1-(1-\rho)c)}{ 1-\rho+\rho\hat{\gamma}(\rho)(1-(1-\rho)c)}
\end{align*}
which we aim to show approaches zero as $\rho\downarrow \max\{0,1-c^{-1}\}$. For this, assume $\rho_k\hat{\gamma}(\rho_k)(1-(1-\rho_k)c)\to \ell\in (0,\infty]$ for a sequence $\rho_k\downarrow \max\{0,1-c^{-1}\}$. Then, from the defining equation of $\hat{\gamma}(\rho)$ in Theorem~\ref{th:Chitour},
\begin{align*}
	1 &= \int \frac{(1-(1-\rho_k)c)t}{\rho_k\hat{\gamma}(\rho_k)(1-(1-\rho_k)c)+(1-\rho_k)(1-(1-\rho_k)c)t}\nu(dt) \\
	&\leq \frac{(1-(1-\rho_k)c) \limsup_N \Vert C_N\Vert}{\rho_k\hat{\gamma}(\rho_k)(1-(1-\rho_k)c)+(1-\rho_k)(1-(1-\rho_k)c)\limsup_N \Vert C_N\Vert} \\
	&\to \frac{\lim_k (1-(1-\rho_k)c) \limsup_N \Vert C_N\Vert}{\ell + \lim_k (1-\rho_k) (1-(1-\rho_k)c) \limsup_N \Vert C_N\Vert} \\
	&< 1
\end{align*}
since the limit is either zero (when $c\geq 1$) or $(1-c)\limsup_N \Vert C_N\Vert/(\ell+(1-c)\limsup_N \Vert C_N\Vert)<1$ (when $c<1$). But this is a contradiction. This implies that $\rho\hat{\gamma}(\rho)(1-(1-\rho)c)\to 0$ and consequently $\hat f(\rho)\to 0$ as $\rho\downarrow \max\{0,1-c^{-1}\}$, which completes the proof for $\hat S(\rho)$.

Similarly, for $\check S(\rho)$, define
\begin{align*}
\check{f} : (0,1] &\to (0,1] \\
\rho &\mapsto \frac{T_\rho}{1-\rho+T_{\rho}}
\end{align*}
where we recall that $T_{\rho}=\rho\check{\gamma}(\rho)F(\check{\gamma}(\rho);\rho)$ and which is such that $\check{S}_N(\rho)=(1-\check{f}(\rho))\frac1n\sum_{i=1}^nz_iz_i^*+\check{f}(\rho)I_N$. We will show that $\check{f}$ is continuously increasing and onto. The continuity arises from the continuity of $\check\gamma$. We first show that $\check\gamma$ is onto. For the upper limit, $\check f(1)=1$. For the lower limit, assume $T_{\rho_k}\to \ell \in (0,\infty]$ over a sequence $\rho_k\to 0$, so that in particular $T_{\rho_k}\rho_k^{-1}\to \infty$. Then, by the definition of $\check{\gamma}(\rho)$ and since $F(x;\rho)=(1-\rho)\frac1{xF(x;\rho)}+\rho-c(1-\rho)$,
\begin{align*}
	1 &= \int \frac{1}{\check{\gamma}(\rho_k)\rho_k t^{-1} + T_{\rho_k} \rho_k^{-1}\frac{1-\rho_k}{1-\rho_k + T_{\rho_k}}}\nu(dt) \to 0
\end{align*}
by dominated convergence (recall that $\nu$ has bounded support), which is a contradiction. This implies $\check f(\rho)\to 0$ as $\rho\to 0$. It remains to show that $\check{f}$ is increasing. For this, we will rewrite the equation defining $\check\gamma(\rho)$ as a function of $\check{f}(\rho)$. Using again $F(x;\rho)=(1-\rho)\frac1{xF(x;\rho)}+\rho-c(1-\rho)$, we first have, for each $t\geq 0$,
\begin{align*}
	\check\gamma(\rho)\rho + \frac{1-\rho}{(1-\rho)c + F(\check\gamma(\rho);\rho)}t &= \check\gamma(\rho)\rho + \frac{1-\rho}{(1-\rho)\frac1{\check\gamma(\rho)F(\check\gamma(\rho);\rho)}+\rho}t \\
	&= \check\gamma(\rho)\rho + \frac{(1-\rho)\check\gamma(\rho)F(\check\gamma(\rho);\rho)}{1-\rho+\rho\check\gamma(\rho)F(\check\gamma(\rho);\rho)}t \\
	&= \frac{\rho\check\gamma(\rho)F(\check\gamma(\rho);\rho)}{F(\check\gamma(\rho);\rho)} + \frac{1-\rho}{\rho} \check{f}(\rho)t \\
	&= \frac1{F(\check\gamma(\rho);\rho)} \frac{(1-\rho)\check{f}(\rho)}{1-\check{f}(\rho)} + \frac{1-\rho}{\rho} \check{f}(\rho)t
\end{align*}
where in the last equality we used $(1-\rho)\check{f}(\rho)=(1-\check{f}(\rho))\rho\check{\gamma}(\rho)F(\check{\gamma}(\rho);\rho)$. We now work on $F(\check\gamma(\rho);\rho)$. By its implicit definition,
\begin{align*}
	\frac1{F(\check\gamma(\rho);\rho)} &= \frac1{(1-\rho)\frac1{\check\gamma(\rho)F(\check\gamma(\rho);\rho)}+\rho-c(1-\rho)} \\
	&= \frac{\rho \check\gamma(\rho)F(\check\gamma(\rho);\rho)}{\rho(1-\rho)+\rho^2\check\gamma(\rho)F(\check\gamma(\rho);\rho)-c(1-\rho)\rho\check\gamma(\rho)F(\check\gamma(\rho);\rho)} \\
	&= \frac{(1-\rho)\check{f}(\rho)}{1-\check{f}(\rho)}\frac1{\rho(1-\rho)+\rho \frac{(1-\rho)\check{f}(\rho)}{1-\check{f}(\rho)}-c(1-\rho)\frac{(1-\rho)\check{f}(\rho)}{1-\check{f}(\rho)}} \\
	&= \frac{\check{f}(\rho)}{\rho-c(1-\rho)\check{f}(\rho)}
\end{align*}
where the last equation follows from standard algebraic simplification. Note here in particular that, by positivity of $F(x;\rho)$ for $x>0$, $\rho-c(1-\rho)\check{f}(\rho)>0$. Plugging the two results above in the defining equation for $\check{\gamma}(\rho)$, we obtain
\begin{align}
	\label{eq:1_of_f}
	1 &= \int \frac{t}{\frac{\check{f}(\rho)}{\rho-c(1-\rho)\check{f}(\rho)}\frac{(1-\rho)\check{f}(\rho)}{\rho(1-\check{f}(\rho))} + \frac{1-\rho}{\rho}\check{f}(\rho)t }\nu(dt).
\end{align}
Now assume that $\check{f}(\rho)$ is decreasing on an open neighborhood of $\rho_0\in(0,1)$. Then $\rho\mapsto \frac{1-\rho}{\rho}\check{f}(\rho)$ and $\rho \mapsto \frac{(1-\rho)\check{f}(\rho)}{\rho(1-\check{f}(\rho))}$ are also decreasing. This follows from the fact that, on this neighborhood, $\rho\mapsto (1-\rho)/\rho=1/\rho-1$, $\rho\mapsto 1-\rho$, and $\rho\mapsto\check{f}(\rho)/(1-\check{f}(\rho))=-1+1/(1-\check{f}(\rho))$ are all positive decreasing functions of $\rho$. Finally,
\begin{align*}
	\frac{\check{f}(\rho)}{\rho-c(1-\rho)\check{f(\rho)}} &= \frac1{\frac{\rho}{\check{f}(\rho)}+c(\rho-1)}
\end{align*}
which is also positive decreasing, since $\rho\mapsto\rho/\check{f}(\rho)$ and $\rho\mapsto c(\rho-1)$ are both increasing and of positive sum. But then, the right-hand side of \eqref{eq:1_of_f} is increasing on a neighborhood of $\rho_0$ while being constant equal to one, which is a contradiction. Therefore, our initial assumption that $\check{f}(\rho)$ is locally decreasing around $\rho_0$ does not hold, and therefore $\check{f}(\rho)$ is increasing there and thus increasing on $(0,1]$. This completes the proof.

\subsection{Proof of Proposition~\ref{prop:shrink}}
\label{sec:proof_optimal}
	We only prove the result for $\hat{C}_N$, the treatment for $\check{C}_N$ being the same. First observe that, denoting $A_N(\hat\rho)=\frac{\hat{C}_N(\hat\rho)}{\frac1N\tr \hat{C}_N(\hat\rho)} - \frac{\hat{S}_N(\hat\rho)}{M_{\hat\mu_{\hat\rho},1}}$,
	\begin{align*}
		&\sup_{\hat\rho\in \hat{\mathcal R}_\varepsilon} \left| \hat{D}_N(\hat\rho) - \frac1N\tr\left( \left( \frac{\hat{S}_N(\hat\rho)}{M_{\hat\mu_{\hat\rho},1}} - C_N \right)^2\right) \right| \\
		&=\sup_{\hat\rho\in \hat{\mathcal R}_\varepsilon} \left| \frac1N\tr \left(A_N(\hat\rho) \left[ \frac{\hat{C}_N(\hat\rho)}{\frac1N\tr \hat{C}_N(\hat\rho)} + \frac{\hat{S}_N(\hat\rho)}{M_{\hat\mu_{\hat\rho},1}} - 2 C_N\right] \right)\right| \\
	&\leq \sup_{\hat\rho\in \hat{\mathcal R}_\varepsilon} \left\{ 2\left| \frac1N\tr (A_N(\hat\rho)C_N) \right| + \left| \frac1N\tr \left(A_N(\hat\rho) \left[ \frac{\hat{C}_N(\hat\rho)}{\frac1N\tr \hat{C}_N(\hat\rho)} + \frac{\hat S_N(\hat\rho)}{M_{\hat\mu_{\hat\rho},1}}  \right]\right) \right|  \right\} \\
		&\leq \sup_{\hat\rho\in \hat{\mathcal R}_\varepsilon} \left\Vert A_N(\hat\rho)\right\Vert \sup_{\hat\rho\in \hat{\mathcal R}_\varepsilon} \left( 3 + \frac{\frac1N\tr\hat S_N(\hat\rho)}{M_{\hat\mu_{\hat\rho},1}} \right)
	\end{align*} 
	where we used $|\tr (AB)|\leq \tr A \Vert B\Vert$ for nonnegative definite $A$ along with $\frac1N\tr C_N=1$. Now, 
	\begin{align*}
		\sup_{\hat\rho\in \hat{\mathcal R}_\varepsilon} \left\Vert A_N(\hat\rho)\right\Vert &\leq  \frac{\sup_{\hat\rho\in \hat{\mathcal R}_\varepsilon} M_{\hat\mu_{\hat\rho},1} \sup_{\hat\rho\in \hat{\mathcal R}_\varepsilon} \Vert \hat{C}_N(\hat\rho)-\hat{S}_N(\hat\rho) \Vert}{\inf_{\hat\rho\in \hat{\mathcal R}_\varepsilon} \frac1N\tr\hat{C}_N(\hat\rho)M_{\hat\mu_{\hat\rho},1} } \nonumber \\
		&+ \frac{ \sup_{\hat\rho\in \hat{\mathcal R}_\varepsilon}  \Vert \hat{S}_N(\hat\rho)\Vert \sup_{\hat\rho\in \hat{\mathcal R}_\varepsilon} \left| \frac1N\tr\hat{C}_N(\hat\rho) - M_{\hat\mu_{\hat\rho},1} \right| }{\inf_{\hat\rho\in \hat{\mathcal R}_\varepsilon} M_{\hat\mu_{\hat\rho},1}\frac1N\tr\left(\hat{C}_N(\hat\rho)\right) }.
	\end{align*}
	Since $M_{\hat\mu_{\hat\rho},1}=\frac1{\hat\gamma(\hat\rho)}\frac{1-\hat\rho}{1-(1-\hat\rho)c}$ is uniformly bounded across $\hat\rho \in \hat{\mathcal R}_\varepsilon$, this finally implies from Theorem~\ref{th:Chitour} and Corollary~\ref{cor:joint} that both right-hand side terms tend almost surely to zero in the large $N,n$ limit (in particular since the denominators are bounded away from zero), and finally
	\begin{align*}
		\sup_{\hat\rho\in \hat{\mathcal R}_\varepsilon} \left| \hat{D}_N(\hat\rho) - \frac1N\tr \left[\left( \frac{\hat{S}_N(\hat\rho)}{M_{\hat\mu_{\hat\rho},1}} - C_N \right)^2\right] \right| &\asto 0.
	\end{align*}

Moreover, from Lemma~\ref{lem:equivalent}, for each $\hat\rho\in(\max\{0,1-c^{-1}\},1]$,
	\begin{align*}
		\frac1N\tr\left[ \left( \frac{\hat{S}_N(\hat\rho)}{M_{\hat\mu_{\hat\rho},1}} - C_N \right)^2\right] = \frac1N\tr\left[ \left( \bar{S}_N(\rho) - C_N \right)^2\right]
	\end{align*}
with $\rho=\hat{\rho}(\frac1{\hat\gamma(\hat\rho)}\frac{1-\hat\rho}{1-(1-\hat\rho)c}+\hat\rho)^{-1}\in (0,1]$ and with $\bar{S}_N=(1-\rho)\frac1n\sum_{i=1}^nz_iz_i^*+\rho I_N$. Also, using $\frac1N\tr \left(\frac1n\sum_{i=1}^nz_iz_i^*\right)\asto M_{\nu,1}=1$, $\frac1N\tr \left[\left(\frac1n\sum_{i=1}^nz_iz_i^*\right)^2\right]\asto M_{\nu,2}+c$, and basic arithmetic derivations
	\begin{align*}
		\sup_{\rho\in [0,1]} \left| \frac1N\tr\left[ \left( \bar{S}_N(\rho) - C_N \right)^2\right] - \bar{D}(\rho) \right| &\asto 0
	\end{align*}
	where 
	\begin{align*}
		\bar{D}(\rho) &= (M_{\nu,2}-1)\rho^2+c(1-\rho)^2.
	\end{align*}
	Note importantly that, from the Cauchy-Schwarz inequality, $1=M_{\nu,1}^2\leq M_{\nu,2}$ and therefore $M_{\nu,2}-1\geq 0$ with equality if and only if $\nu={\bm\delta}_a$ for some $a\geq 0$ almost everywhere.
	From the above convergence, we then have, for any $\varepsilon>0$ small,
	\begin{align}
		\label{eq:sup_rho0}
		\sup_{\hat\rho\in \hat{\mathcal R}_\varepsilon} \left| \hat{D}_N(\hat\rho) - \bar{D}(\rho)\right| &\asto 0.
	\end{align}

	Now, call $\rho^\star$ the minimizer of $\bar{D}(\rho)$ over $[0,1]$. It is easily verified that $\rho^\star\in(0,1]$ is as defined in the theorem. Also denote $\hat\rho^\star$ the unique value such that $\rho^\star={\hat{\rho}^\star}(\frac1{\hat\gamma(\hat\rho^\star)}\frac{1-\hat\rho^\star}{1-(1-\hat\rho^\star)c}+\hat\rho^\star)^{-1}$, which is well defined according to Lemma~\ref{lem:equivalent}. Call also $\hat\rho^\circ_N$ the minimizer of $\hat{D}_N(\hat\rho)$ over $\hat{\mathcal R}_\varepsilon$ and $\rho^\circ_N=\hat{\rho}^\circ_N(\frac1{\hat\gamma(\hat\rho^\circ_N)}\frac{1-\hat\rho^\circ_N}{1-(1-\hat\rho^\circ_N)c}+\hat\rho^\circ_N)^{-1}$. If $\varepsilon$ is as given in the theorem statement, $\hat{\rho}^\star\in \hat{\mathcal R}_\varepsilon$ and then
	\begin{align*}
		\bar{D}(\rho^\star) &\leq \bar{D}(\rho^\circ_N) \\
		\hat{D}_N(\hat\rho^\circ_N) &\leq \hat{D}_N(\hat\rho^\star) \\
		\hat{D}_N(\hat\rho^\star) - \bar{D}(\rho^\star) &\asto 0 \\
		\hat{D}_N(\hat\rho^\circ_N) - \bar{D}(\rho^\circ_N) &\asto 0
	\end{align*}
	the last two equations following from \eqref{eq:sup_rho0} (the joint convergence in \eqref{eq:sup_rho0} is fundamental since $\rho^\circ_N$ and $\hat\rho^\circ_N$ are not constant with $N$). These four relations together ensure that
	\begin{align*}
		\hat{D}_N(\hat\rho^\circ_N) - \bar{D}(\rho^\star) &\asto 0 \\
		\hat{D}_N(\hat\rho^\circ_N) - \hat D_N(\hat\rho^\star) &\asto 0.
	\end{align*}
	These and the fact that $\bar{D}(\rho^\star)=D^\star$ as defined in the theorem statement conclude the proof of the first part of the theorem.

	For the second part, denoting $\rho_N={\hat{\rho}_N}(\frac1{\hat\gamma(\hat\rho_N)}\frac{1-\hat\rho_N}{1-(1-\hat\rho_N)c}+\hat\rho_N)^{-1}$, we have that $\bar{D}(\rho_N)-\bar{D}(\rho^\star)\asto 0$ by continuity of $\bar{D}$ since $\rho_N\asto \rho^\star$ and therefore, since $\hat{D}_N(\hat\rho_N)-\bar{D}(\rho_N)\asto 0$ by \eqref{eq:sup_rho0}, $\hat{D}_N(\hat\rho_N)-\bar{D}(\rho^\star)\asto 0$ which is the expected result.

\subsection{Proof of Proposition~\ref{prop:optimal_shrink}}
\label{sec:proof_optimal_shrink}
	We first show the following identities
	\begin{align}
		\frac1n\tr\left[ \left( \frac1n\sum_{i=1}^n \frac{x_ix_i^*}{\frac1N\Vert x_i\Vert^2} \right)^2\right] - c_N  &\asto M_{\nu,2} \label{eq:est_3} \\
		\sup_{\rho\in \check{\mathcal R}_\varepsilon} \left| T_{\rho} - \rho \frac1n \sum_{i=1}^n \frac{x_i^*\check{C}_N(\rho)^{-1}x_i}{\Vert x_i\Vert^2} \right| &\asto 0 \label{eq:est_2}.
\end{align}
Identity \eqref{eq:est_3} unfolds from $\frac1n\tr \left[( \frac1n\sum_{i=1}^n z_iz_i^*)^2\right]\asto M_{\nu,2}+cM_{\nu,1}^2=M_{\nu,2}+c$ and from $\max_{1\leq i\leq n}|\frac1N\Vert z_i\Vert^2-1|\asto 0$. As for Equation \eqref{eq:est_2}, it is a consequence of the elements of the proof of Theorem~\ref{th:Wiesel}. Indeed, from \eqref{eq:BN_Bi},
\begin{align*}
	\rho \frac1Nx_i^* \check{C}_N(\rho)^{-1}x_i &=\rho\frac1Nx_i^* \check{B}_{(i)}(\rho)^{-1}x_i\left(\frac1N\tr\check{B}_N(\rho)-c_N(1-\rho)\right)
\end{align*}
where $\check{B}_{(i)}(\rho)=\check{B}_N(\rho)-\frac1n\frac{1-\rho}{\frac1N\tr \check{B}_N}\frac{x_ix_i^*}{\frac1Nx_i^*\check{B}_N(\rho)^{-1}x_i}$, which according to \eqref{eq:B_FN} further reads
\begin{align*}
	\rho \frac1Nx_i^* \check{C}_N(\rho)^{-1}x_i &=\rho\frac1Nx_i^* \check{B}_{(i)}(\rho)^{-1}x_iF_N\left(\left[\frac1{n}\sum_{i=1}^n \frac{\Vert x_i\Vert^2}{\frac1Nx_i^*\check{B}_{(i)}(\rho)^{-1}x_i} \right]^{-1};\rho\right)
\end{align*}
with $F_N(x;\rho)$ the same function as $F$ but with $c_N$ in place of $c$ (recall that in \eqref{eq:B_FN}, $\check{d}_i=\frac1Nz_i^*\check{B}_{(i)}(\rho)^{-1}z_i$). Since the $\tau_i$ normalization is irrelevant in the expression above, $x_i$ can be replaced by $z_i$. Using the convergence result \eqref{eq:conv_checkd} and the continuity and boundedness of $x\mapsto xF_N(x)$, we then have
\begin{align*}
	\sup_{\rho\in\check{\mathcal R}_\varepsilon} \max_{1\leq i\leq n} \left| \rho \frac1Nz_i^* \check{C}_N(\rho)^{-1}z_i - \rho\check{\gamma}(\rho)F(\check{\gamma}(\rho);\rho) \right| \asto 0.
\end{align*}
As a consequence,
\begin{align*}
	&\sup_{\rho\in\check{\mathcal R}_\varepsilon} \left| \rho \frac1n\sum_{i=1}^n \frac1Nz_i^* \check{C}_N(\rho)^{-1}z_i - \rho\check{\gamma}(\rho)F(\check{\gamma}(\rho);\rho) \right| \nonumber \\
	&\leq \sup_{\rho\in\check{\mathcal R}_\varepsilon} \max_{1\leq i\leq n} \left| \rho \frac1Nz_i^* \check{C}_N(\rho)^{-1}z_i - \rho\check{\gamma}(\rho)F(\check{\gamma}(\rho);\rho) \right| \\
	&\asto 0.
\end{align*}
This, and the fact that $\max_{1\leq i\leq n}|\frac1N\Vert z_i\Vert^2 - 1| \asto 0$ gives the result.

It remains to prove that $\hat\rho_N\asto \hat\rho^\star$ and $\check\rho_N\asto \check\rho^\star$. We only prove the first convergence, the second one unfolding along the same lines. First observe from Corollary~\ref{cor:joint} that the defining equation of $\hat\rho_N$ implies
\begin{align*}
	\hat{f}(\hat\rho_N)&= \frac{c}{M_{\nu,2}+c-1} + \ell_n
\end{align*}
for some sequence $\ell_n\asto 0$, with $\hat f:x\mapsto x(\frac1{\hat\gamma(x)}\frac{1-x}{1-(1-x)c}+x))^{-1}$. Since $\hat f$ is a one-to-one growing map from $(\max\{0,1-c^{-1}\},1]$ onto $(0,1]$ (Lemma~\ref{lem:equivalent}) and $\frac{c}{M_{\nu,2}+c-1}\in (0,1)$, such a $\hat\rho_N$ exists (not necessarily uniquely though) for all large $N$ almost surely. Taking such a $\rho_N$, by definition of $\hat\rho^\star$, we further have
\begin{align*}
	\hat{f}(\hat\rho_N) - \hat f(\hat\rho^\star) &\asto 0
\end{align*}
which, by the continuous growth of $\hat f$, ensures that $\hat\rho_N\asto \hat\rho^\star$. The convergence $\hat D_N(\hat\rho_N)\asto D^\star$ is then an application of Proposition~\ref{prop:shrink}.

\bibliographystyle{elsarticle-harv}
\bibliography{/home/romano/Documents/PhD/phd-group/papers/rcouillet/tutorial_RMT/book_final/IEEEabrv.bib,/home/romano/Documents/PhD/phd-group/papers/rcouillet/tutorial_RMT/book_final/IEEEconf.bib,/home/romano/Documents/PhD/phd-group/papers/rcouillet/tutorial_RMT/book_final/tutorial_RMT.bib,/home/romano/Documents/work/papers/robust_est_elliptic/robust_est.bib}

\end{document}